\documentclass[11pt,a4paper]{article}

\usepackage[T1]{fontenc}
\usepackage[utf8]{inputenc}
\usepackage[english]{babel}

\usepackage{amsmath,amssymb,amsfonts,amsthm}
\usepackage{mathtools}
\usepackage{bbm}
\usepackage{listings}
\usepackage[font=footnotesize,labelfont=bf]{caption}
\usepackage{authblk}

\usepackage[a4paper,margin=3cm]{geometry}

\usepackage[colorlinks=true,linkcolor=blue,citecolor=blue,urlcolor=blue]{hyperref}
\newcommand{\field}[1]{\mathbb{#1}}
\newcommand{\N}{\field{N}}
\newcommand{\R}{\field{R}}

\DeclareMathOperator*{\argmax}{arg\,max}
\newtheorem{theorem}{Theorem}[section]
\newtheorem{lemma}[theorem]{Lemma}

\newtheorem{definition}[theorem]{Definition}
\newtheorem{remark}[theorem]{Remark}
\newtheorem{assumption}[theorem]{Assumption}
\newtheorem{algorithm}[theorem]{Algorithm}

\title{Data-Driven Subsampling Rates for Diffusion Parameter Estimation of SDEs}

\author[1]{Felix Lindner}
\author[2]{Andre Schmei\ss{}er}
\author[1]{Felipe Trolldenier}
\author[3]{Raimund Wegener}

\affil[1]{Universit\"at Kassel, Institut f\"ur Mathematik, Heinrich-Plett-Str. 40, D-34132 Kassel, Germany}
\affil[2]{Fraunhofer ITWM, Fraunhofer Platz 1, D-67663 Kaiserslautern, Germany}
\affil[3]{Petrusstr. 1, D-54292 Trier, Germany}

\date{}

\begin{document}

\maketitle

\begin{abstract}
We study the problem of diffusion parameter estimation for stochastic differential equation (SDE) models in scenarios where data and model are compatible only on specific scales that have yet to be determined. We introduce a simple and efficient method for selecting suitable rates at which given time series data should be subsampled in order to ensure that the statistical structure of the subsampled data is consistent with the behavior of the SDE model on an infinitesimal scale. Our approach is based on analyzing the statistics of the lengths of monotonically increasing or decreasing segments in the subsampled data sequence, which we refer to as monotone runs. As an analytical foundation, we prove for a large class of SDEs with additive noise that the lengths of monotone runs at an infinitesimal scale are approximately geometrically distributed with success probability $1/2$. This universal characterization is employed to derive an automated method for selecting appropriate subsampling rates for given time series data that is directly applicable in real-world scenarios and does not rely on an asymptotic framework of multiscale diffusions. The approach is demonstrated using an application from industrial mathematics concerning surrogate models for fiber lay-down curves in production processes of nonwoven textiles.
\end{abstract}

\vspace{0.5em}


\noindent\textbf{AMS-Classification:} 60H10, 62M05, 62F99, 34E13

\section{Introduction}

Stochastic differential equation (SDE) models are often applied in contexts where data and model are compatible only on specific scales. Examples include molecular dynamics, ocean-atmosphere sciences and high frequency data in finance, see, e.g., \cite{Pavliotis2007,Monahan2011,Mykland2012}. 
In this article we consider SDEs with additive noise of the form
\begin{equation}
\label{eqGenSDEIntro}
\text{d}\mathbf{X}_t = \mathbf{F}(\mathbf{X}_t) \, \text{d}t + \mathbf{A} \, \text{d}\mathbf{W}_t, 
\end{equation}
$t \in \R_0^+ \coloneq [0,\infty)$, with  initial condition $\mathbf{X}_0 \in \R^d$, drift function $\mathbf{F} = (F_1,\ldots,F_d) \colon \R^d \rightarrow \R^d$, diffusion matrix $\mathbf{A} = (A_{ij})_{i,j} \in \R^{d \times m}$, and $m$-dimensional standard Brownian motion $\mathbf{W}_t = (W_t^{(1)}, \ldots, W_t^{(m)})$. Let $X_t$ denote a fixed component of the solution process $\mathbf{X}_t = (X_t^{(1)}, \ldots, X_t^{(d)})$, say $X_t = X_t^{(d)}$ for definiteness. By introducing $A \coloneq (\sum_{j = 1}^m A_{dj}^2)^{1/2}$, where we assume $A > 0$, we can represent $X_t$ in terms of a one-dimensional standard Brownian motion $W_t \coloneq A^{-1} \sum_{j = 1}^m A_{dj} W_t^{(j)}$, i.e.,
\begin{equation}
\label{eqRelevantSDEIntro}
\text{d}X_t = F(\mathbf{X}_t) \, \text{d}t + A \, \text{d} W_t,
\end{equation}
where $F \coloneq F_d$. Our goal is to estimate the diffusion constant $A$ based on available data. While the issue of data-model compatibility being restricted to specific scales is typically formalized in a framework of multiscale diffusions, the estimation method presented in this article does not assume such a setting and can be applied directly to real-world data. Nevertheless, we make use of a multiscale framework where appropriate to illustrate the method.

Multiscale diffusions consist of a fast and a slow variable representing small- and large-scale structures, respectively, with the separation of their characteristic time scales described by a parameter $\varepsilon$. In this framework, we consider the multiscale diffusion as a generator of synthetic data, while the coarse-grained limit for $\varepsilon \rightarrow 0$ of the slow variable is the model of interest given by \eqref{eqGenSDEIntro}. 
Typically, \textit{averaging} and \textit{homogenization} of SDEs are discussed in literature. For the homogenization setting, \cite{Pavliotis2007,Papavasiliou2009} have shown that classical estimators fail due to being asymptotically biased. Furthermore, the authors of the aforementioned studies have proven that, asymptotically as $\varepsilon \rightarrow 0$, the bias can by reduced by subsampling the data at an appropriate rate that depends on $\varepsilon$, the specific model, and the parameters to be estimated. The asymptotic results for Gaussian processes in \cite{Azencott2010,Azencott2011} also show the dependence of the asymptotic optimal subsampling rate on $\varepsilon$. More general results involving non-Gaussian processes can be found in \cite{Azencott2015}. However, even in those cases where the available data can be assumed to originate from a suitable multiscale diffusion, the scale separation parameter is typically unknown and therefore it is not possible to subsample appropriately without additional information. Because of this, different approaches that avoid subsampling can be found in literature. In \cite{Crommelin2011} an approach using eigenpairs of the diffusion operator and those of the conditional expectation operator is presented. An estimation procedure based on the Dynkin formula has been proposed in \cite{Krumscheid2013,Kalliadasis2015,Krumscheid2018}. The extrema quadratic variation estimator in \cite{Manikas2018,Manikas2019} is based on the quadratic variation induced by an extrema partition and is asymptotically unbiased for decreasing step sizes and $\varepsilon \rightarrow 0$ in the homogenization setting. 

In contrast, we suggest a data-driven method to identify an appropriate scale on which the statistical structure of given time series data is compatible with the behavior of the SDE model on an infinitesimal scale. For this, our approach draws on the original idea of subsampling and supplements it with an indicator that can be used to determine the range of suitable subsampling rates based on the data set at hand. Let $(X_{n \delta})_{n \in \N_0}$ be a discrete-time sample of the process defined by \eqref{eqRelevantSDEIntro} with step size $\delta > 0$. By \textit{monotone runs} we denote segments of consecutive increments $X_{n \delta} - X_{(n-1) \delta}$, $X_{(n+1) \delta} - X_{n \delta}, \ldots$ with identical signs. We show that the length of the first monotone run at infinitesimal scale is geometrically distributed with parameter $1/2$, i.e.,
\begin{equation*}
\lim_{\delta \rightarrow 0} P \bigl( \min \left\lbrace n \in \N: (X_{n \delta} - X_{(n-1) \delta}) (X_{(n+1) \delta} - X_{n \delta}) < 0 \right\rbrace = k \bigr) = \frac{1}{2^k},
\end{equation*}
$k \in \N$. Furthermore, under an ergodicity assumption, we prove that the mean length of monotone runs converges almost surely to $2$, which is the value we expect due to the mentioned convergence of the distribution. In order to find the subsampling rates for which given time series data is most likely to originate from the model \eqref{eqRelevantSDEIntro} at infinitesimal scale, our idea is to use the results on the asymptotic distribution by comparing them with corresponding characteristics for the lengths of monotone runs in various subsamples of the available data. As a suitable indicator, our main focus is on the mean length of monotone runs.

Let $x_{0}, x_{\Delta}, \ldots, x_{N \Delta}$, $N \in \N$, be time series data with constant step size $\Delta > 0$. We are looking for a subsampling factor $k\ll N$ such that the statistical properties of $(x_{n k \Delta})_{n \in \lbrace 0,\ldots,\lfloor N/k \rfloor \rbrace}$ match $(X_{n \delta})_{n \in \N_0}$ for small $\delta > 0$. For each fixed subsampling factor $k \ll N$, the given data set can be partitioned into $k$ disjoint subsets of subsampled data points with step size $k \Delta$, which together contain a total of $N+1-k$ increments. Let $M(k)$ be the total number of monotone runs in all $k$ subsets. Then the mean length of the monotone runs as a function of the subsampling factor $k$ is given by
\begin{equation*}
L(k) \coloneq \frac{N+1-k}{M(k)}.
\end{equation*}
For each subsampling factor $k$, the estimate
\begin{equation*}
\hat{A}(k) \coloneq \sqrt{\frac{\sum_{n=0}^{N-k} (x_{(n+k) \Delta}-x_{n \Delta})^2}{k \Delta (N+1-k)}}
\end{equation*}
for the diffusion constant $A$ can be calculated based on the increments of the $k$ subsets. Both quantities are plotted in dependence on the effective step size $k \Delta$ for some examples in Figure \ref{figIntroAAMon}. The curves of the estimated values for the diffusion constant $A$ can vary significantly, but typically drop for decreasing effective step sizes due to the multiscale effect described above. In any case, the curve of the mean length of monotone runs serves as an indicator for suitable step sizes $k \Delta$, as these are characterized by a mean length of $2$. In all three cases in Figure \ref{figIntroAAMon}, the combination of the two curves shows that $A$ would be underestimated without subsampling, while depending on the scenario, a different effective step size is appropriate. Since our method applies directly to data with fixed step size $\Delta$ and requires no information about $\varepsilon$, it is entirely data-driven and can be applied without relying on scale-separation limits.

\begin{figure}[t]
\begin{minipage}[h]{0.32\linewidth}
\begin{center}
\includegraphics[width=1\linewidth]{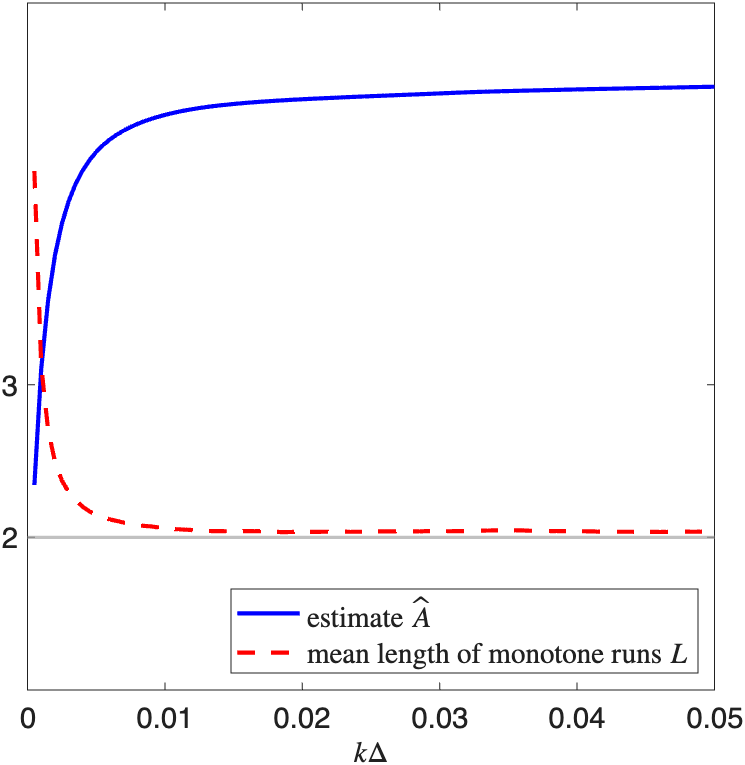} 
\end{center}
\end{minipage}
\hfill
\begin{minipage}[h]{0.32\linewidth}
\begin{center}
\includegraphics[width=1\linewidth]{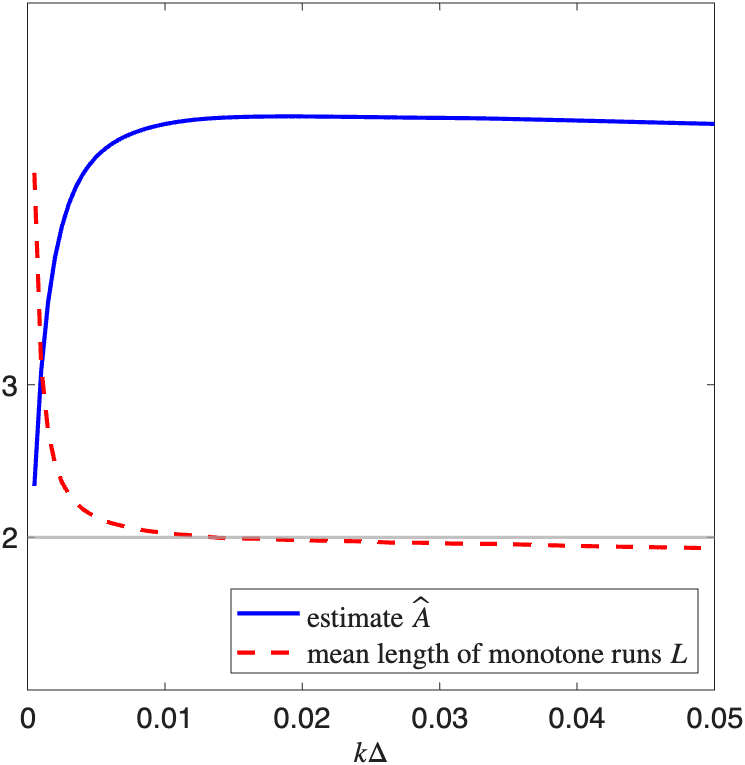} 
\end{center}
\end{minipage}
\hfill
\begin{minipage}[h]{0.32\linewidth}
\begin{center}
\includegraphics[width=1\linewidth]{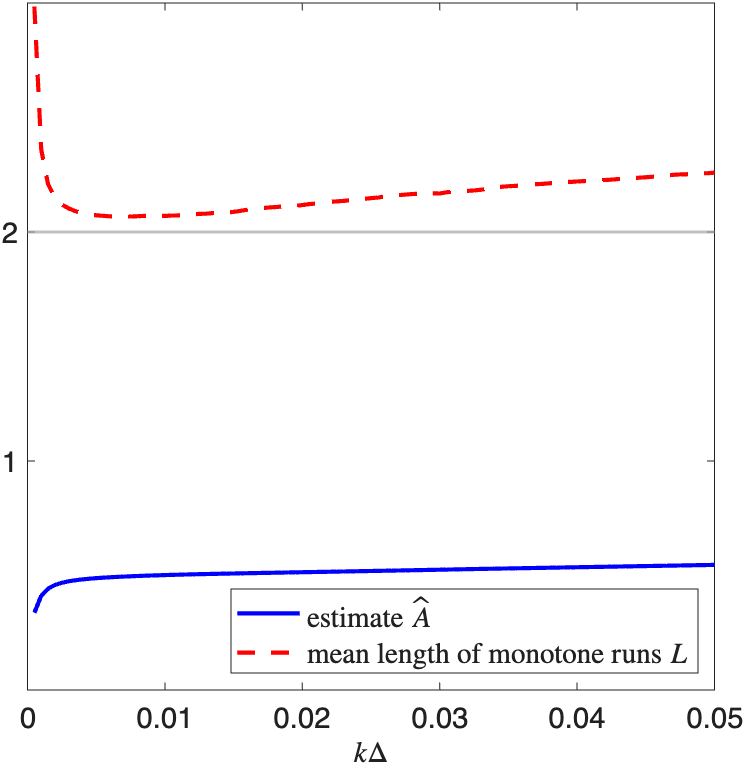} 
\end{center}
\end{minipage}
\caption{Three examples for the mean length of monotone runs and the estimate for the diffusion constant based on data with step size $\Delta$ in dependence on the effective step size $k \Delta$ when subsampling with factor $k$.}
\label{figIntroAAMon}
\end{figure}

For the demonstration of our method we use an example from industrial mathematics. In order to optimize the design of the production process of nonwoven webs of fibers, mathematical models of varying complexity are used for the simulation of fiber lay-down curves, for details see, e.g., \cite{Klar2009}. For the efficient simulation of a large number of such lay-down curves, an SDE model was introduced in \cite{Goetz2007}, for which an example of a simulated fiber can be found in Figure \ref{figIntroFiber}. To apply this model, its parameters must be estimated based on real-world data or simulated data coming from a computationally costly PDE model, both of which have been found to be compatible with the SDE only on specific scales. In addition, a suitable multiscale framework for this SDE model can be defined, similar to applications in molecular dynamics or ocean-atmosphere, for example.

\begin{figure}[ht]
\begin{minipage}[h]{0.48\linewidth}
\begin{center}
\includegraphics[width=1\linewidth]{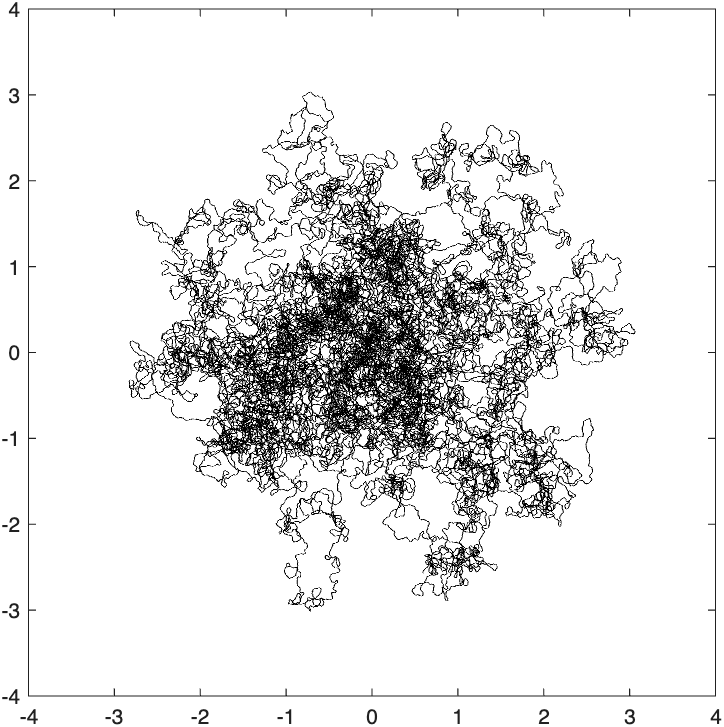} 
\end{center} 
\end{minipage}
\hfill
\begin{minipage}[h]{0.48\linewidth}
\begin{center}
\includegraphics[width=1\linewidth]{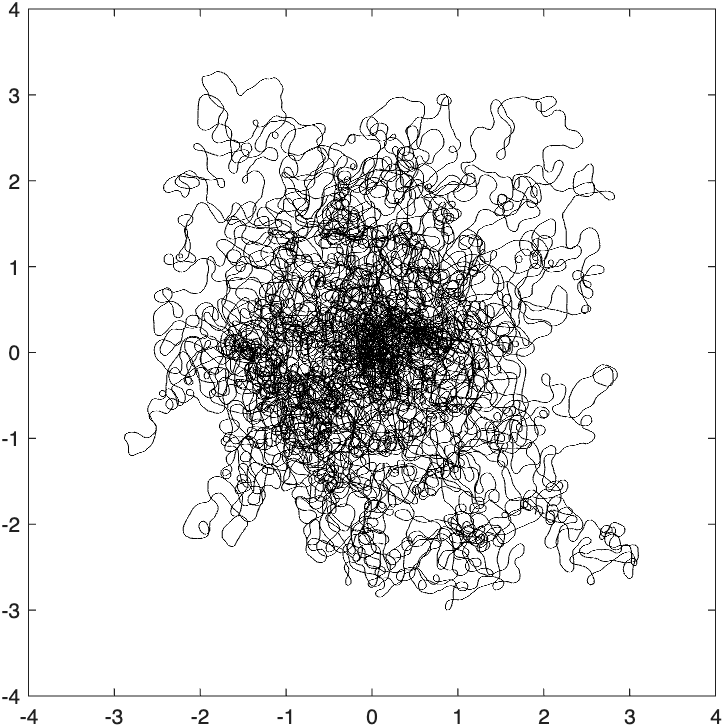} 
\end{center}
\end{minipage}
\caption{Fiber trajectory of \eqref{eqIntroBasicModel} (left) and \eqref{eqIntroSmoothModel} (right) for $V(x_1, x_2) = (x_1^2 + x_2^2)/2$, $A = 5$, $R = K = \sqrt{5}$, $v = 0.1$ and $\varepsilon = 0.2$.}
\label{figIntroFiber}
\end{figure}
The typical SDE model for the lay-down curves of filaments has the form 
\begin{align}
\text{d} \boldsymbol \xi_t &= \boldsymbol \tau (\alpha_t) \: \text{d}t + v \mathbf{e}_1 \: \text{d}t \nonumber \\
\text{d} \alpha_t &= - \nabla V(\boldsymbol \xi_t) \cdot \boldsymbol \tau^\bot (\alpha_t) \: \text{d}t + A \: \text{d} W_t,
\label{eqIntroBasicModel}
\end{align}
with initial conditions $\boldsymbol \xi_0 \in \R^2$ and $\alpha_0 \in [0, 2\pi]$. By $\boldsymbol\xi \colon \R_0^+ \times \Omega \rightarrow \R^2$ the fiber is modeled with respect to a reference curve specifying the movement of the conveyer belt. The normalized tangent on the fiber is described by $\boldsymbol\tau(\alpha) \coloneqq (\cos \alpha, \sin \alpha)$, so that $\boldsymbol\tau^{\bot}(\alpha) \coloneqq (-\sin \alpha, \cos \alpha)$ is the corresponding orthogonal polar unit vector. The potential $V \colon \R^2 \rightarrow \R$ sets the spread of the fiber. A typical choice is 
\begin{equation*}
V(\boldsymbol\xi) = \frac{1}{2} \boldsymbol\xi^\top \mathbf{C}^{-1} \boldsymbol\xi,
\end{equation*}
with positive-definite matrix $\mathbf{C} \in \R^{2 \times 2}$, giving a standard model for the buckling behavior. Turbulence effects in the production process are represented by a one-dimensional standard Brownian motion $W$ with diffusion constant $A \in (0, \infty) \eqcolon \R^+$. In this model, the conveyer belt moves with constant speed, which is set in relation to the production speed resulting in $v \in [0,1]$. Using the notation of \eqref{eqGenSDEIntro} and $\boldsymbol\xi_t = (\xi_t^{(1)}, \xi_t^{(2)})$, we have $\mathbf{X}_t = (\xi_t^{(1)}, \xi_t^{(2)}, \alpha_t)$,
\begin{equation*}
\mathbf{F}(\mathbf{x}) = \begin{pmatrix}
\cos(x_3) + v \\
\sin(x_3) \\
-\nabla V(x_1, x_2) \cdot (-\sin(x_3), \cos(x_3))^\top
\end{pmatrix},
\end{equation*}
$\mathbf{x} = (x_1, x_2, x_3)$, and $\mathbf{A} = \text{diag}(0,0,A)$.
In practical application, this model is adapted to given data by estimating the diffusion constant $A$ and the parameters of the potential $V$ while the speed ratio $v$ is known. 
In addition to the investigation of this real-world situation, we can also supplement the target model with a suitable multiscale diffusion, making it a typical example of homogenization. As shown in \cite{Herty2009}, the model \eqref{eqIntroBasicModel} can be considered as limit for $\varepsilon \rightarrow 0$ of
\begin{align}
\text{d} \boldsymbol\xi_t^{(\varepsilon)} &= \boldsymbol\tau(\alpha_t^{(\varepsilon)}) \: \text{d}t +v \mathbf{e}_1 \: \text{d}t \nonumber \\ 
\text{d} \alpha_t^{(\varepsilon)} &= -\nabla V(\boldsymbol\xi_t^{(\varepsilon)}) \cdot \boldsymbol\tau^\perp(\alpha_t^{(\varepsilon)}) \: \text{d}t + \frac{\kappa_t^{(\varepsilon)}}{\varepsilon} \: \text{d}t \nonumber \\
\text{d} \kappa_t^{(\varepsilon)} &= -\frac{1}{\varepsilon^2 R} \kappa_t^{(\varepsilon)} \: \text{d}t + \frac{K}{\varepsilon} \: \text{d} W_t, \label{eqIntroSmoothModel}
\end{align}
$t \in \R_0^+$, equipped with appropriate initial conditions $\boldsymbol\xi_0 \in \R^2$, $\alpha_0 \in [0, 2\pi]$ and $\kappa_0 \in \R$. Here, $R ,K\in \R^+$ are some constants. In addition to the process $\boldsymbol\xi$ describing the fiber with respect to a reference curve for the conveyer belt movement and the angles $\alpha$, this model is formulated in terms of the curvature $\kappa$ of the fiber. For the investigation of our method we can use an Euler approximation of \eqref{eqIntroSmoothModel} to generate data for which we know the true value of $A$, at least as a limit for $\varepsilon \rightarrow 0$, and which also has the incompatibility with \eqref{eqIntroBasicModel} on small scales typical for real-world data. An example of a fiber trajectory derived from an Euler approximation of \eqref{eqIntroSmoothModel} can be found in Figure \ref{figIntroFiber}.

The article is structured as follows. In Section \ref{SectionAnalyticalResults}, we deal with the analytical background of monotone runs of discrete-time samples in the SDE model \eqref{eqRelevantSDEIntro}. We show that the distribution of the length of the first monotone run for infinitesimal time steps is geometric with success probability $1/2$ and that, under an ergodicity assumption, the mean value of monotone runs along a path is a suitable approximation of the value 2 that we expect due to the asymptotic distribution. These main results are formulated in Subsection \ref{SubsectionMainresults} and the corresponding proofs can be found in Subsection \ref{SubsectionProofs}. We build on this theoretical foundation in Section \ref{SectionTheEstimationMethod}, where we describe our method in detail and show how we select a suitable scale based on available data and calculate an estimate for the diffusion constant of the model. Section \ref{SectionApplication} provides the application to the simulation of fiber lay-down in industrial nonwoven production. In Subsection \ref{SubsectionIntroApplication} we give an introduction to the models that are used in this context and their embedding into the previously established framework. For the numerical tests in \ref{SubsectionNumericalTests}, we generate synthetic data based on the homogenization setting for the stochastic fiber lay-down models. Thus, we know the limit of the true diffusion constant and can demonstrate how our method identifies suitable scales and provides good estimates for the diffusion constant. The complete procedure for the estimation of all parameters in the stochastic fiber lay-down model based on simulated fiber data is investigated in Subsection \ref{SubsectionCompleteProcedure}.

\section{Analytical results}
\label{SectionAnalyticalResults}

Consider the $d$-dimensional SDE model \eqref{eqGenSDEIntro}, i.e.,
\begin{equation*}
\text{d}\mathbf{X}_t = \mathbf{F}(\mathbf{X}_t) \, \text{d}t + \mathbf{A} \, \text{d}\mathbf{W}_t,
\end{equation*}
$t \in \R_0^+$, and its $d$-th component \eqref{eqRelevantSDEIntro}, i.e.,
\begin{equation*}
\text{d} X_t = F(\mathbf{X}_t) \, \text{d}t + A \, \text{d} W_t.
\end{equation*}

Throughout this paper, all considered random variables and stochastic processes are assumed to be defined on a common underlying probability space $(\Omega, \mathcal{A}, P)$. We specify and summarize our assumptions as follows.

\begin{assumption}
\label{Assumptions}
The drift function $\mathbf{F} = (F_1,\ldots,F_d) \colon \R^d \rightarrow \R^d$ in \eqref{eqGenSDEIntro} is locally Lipschitz continuous and such that its $d$-th component $F = F_d$ is of polynomial growth, i.e., there exist constants $C,p \in \R^+$ such that for all $\mathbf{x} \in \R^d$ it holds that
\begin{equation}
\label{eqPolynGrowthCond}
\vert F (\mathbf{x}) \vert \leq C (1 + \Vert \mathbf{x} \Vert^p).
\end{equation}
The diffusion matrix $\mathbf{A} = (A_{ij})_{ij} \in \R^{d \times m}$ in  \eqref{eqGenSDEIntro} is such that the constant $A = (\sum_{j = 1}^m A_{dj}^2)^{1/2}$ in \eqref{eqRelevantSDEIntro} is strictly positive. For every initial random variable $\mathbf{X}_0 \colon \Omega \rightarrow \R^d$ that is independent of the driving Brownian motion $\mathbf{W} = (\mathbf{W}_t)_{t \geq 0}$, there exists a stochastic process $\mathbf{X} = (\mathbf{X}_t)_{t \geq 0}$ with continuous sample paths that solves \eqref{eqGenSDEIntro} in the sense that for every $t \in \R_0^+$ it holds $P$-a.s.\ that
\begin{equation*}
\mathbf{X}_t  = \mathbf{X}_0 + \int_0^t \mathbf{F}(\mathbf{X}_s) \, \text{d}s + \mathbf{A} \, \mathbf{W}_t.
\end{equation*}
\end{assumption}

Recall that $X_t$ and $W_t$ in \eqref{eqRelevantSDEIntro} denote the $d$-th component of $\mathbf{X}_t$ and the one-dimensional Brownian motion given by $W_t = A^{-1} \sum_{j = 1}^m A_{dj} W_t^{(j)}$, respectively. In the special case of a deterministic initial condition $\mathbf{X}_0 = \mathbf{x}$ $P$-a.s., $\mathbf{x} \in \R^d$, we denote the associated solution process of \eqref{eqGenSDEIntro} by $\mathbf{X}^{\mathbf{x}} = (\mathbf{X}^{\mathbf{x}}_t)_{t \geq 0}$ and its $d$-th component by $X^{\mathbf{x}} = (X^{\mathbf{x}}_t)_{t \geq 0}$.

As described in the introduction, our approach is based on the investigation of monotone runs of discrete-time samples of the process $X$.
For step size $\delta >0$ we define the increments
\begin{equation*}
\Delta X_{j,\delta} \coloneq X_{j\delta} - X_{(j-1)\delta},
\end{equation*}
$j\in \N$. Let 
\begin{equation*}
R_{1,\delta} (X) \coloneq \min\lbrace n \in \N: \Delta X_{n+1,\delta} \Delta X_{n,\delta} < 0 \rbrace,
\end{equation*}
where we set $\min \emptyset \coloneq \infty$, be the first occurrence time of a sign change in the increments. The subsequent sign changes are formalized analogously by defining the $j$-th occurrence time
\begin{equation*}
R_{j,\delta} (X) \coloneq \min\lbrace n \in \N, n > R_{j-1,\delta}(X): \Delta X_{n+1,\delta} \Delta X_{n,\delta} < 0 \rbrace,
\end{equation*}
$j \in \N_{\geq 2}$. We set the first recurrence time $T_{1,\delta} (X)$ to be equal to the first occurrence time. In other words, $T_{1,\delta} (X) = R_{1,\delta} (X)$ denotes the index of the first sign change in the present increments and therefore also the length of the first monotone run by counting the increments involved. The length of the $j$-th monotone run is then given by the $j$-th recurrence time
\begin{equation*}
T_{j, \delta}(X) \coloneq R_{j, \delta}(X) - R_{j-1, \delta}(X),
\end{equation*}
for $j \in \N_{\geq 2}$.
If it is clear from the context, we also write $\Delta X_j$, $R_j$ and $T_j$ instead of $\Delta X_{j,\delta}$, $R_{j,\delta}(X)$ and $T_{j,\delta}(X)$, respectively. 

\subsection{Probability and frequency distribution of the lengths of monotone runs}
\label{SubsectionMainresults}

Our first main result concerns the asymptotic distribution of the length of the first monotone run. 

\begin{theorem}[Probability distribution of the length of the first monotone run]
\label{Thm1stResult}
Suppose that Assumption \ref{Assumptions} holds, let $\mathbf{X}_0$ be an initial random variable independent of $\mathbf{W}$, and let $\mathbf{X} = (\mathbf{X}_t)_{t \geq 0}$ be the corresponding solution to \eqref{eqGenSDEIntro} with $d$-th component $X = (X_t)_{t \geq 0}$. Then the probability distribution of the length of the first monotone run of the sampled process $(X_{n\delta})_{n \in \N_0}$ converges to the geometric distribution with parameter $1/2$ as $\delta \rightarrow 0$ , i.e., for all $\ell \in \N \cup \lbrace \infty \rbrace$ it holds that
\begin{equation*}
\lim_{\delta \rightarrow 0} P( T_{1,\delta}(X) = \ell) = \frac{1}{2^\ell},
\end{equation*}
where we set $1/2^\infty \coloneq 0$.
\end{theorem}

Under additional assumptions on the SDE \eqref{eqGenSDEIntro}, it is possible to show that the expectation of the length of the first monotone run converges to $2$ as $\delta \rightarrow0$, i.e.,
\begin{equation*}
\lim_{\delta \rightarrow 0} \text{E}(T_{1,\delta}(X)) = 2.
\end{equation*}
For example, if $F$ is globally bounded, bounds of the form 
\begin{equation*}
2 \Phi\biggl( -\sqrt{\delta} \tfrac{\Vert F \Vert_\infty}{A} \biggr)^{\ell+1} \leq P(T_{1,\delta}(X) = \ell) \leq 2 \Phi\biggl( \sqrt{\delta} \tfrac{\Vert F \Vert_\infty}{A} \biggr)^{\ell+1},
\end{equation*} 
$\ell \in \N$, $\delta > 0$, can be derived directly using arguments similar to those in the proof of Lemma \ref{Lemma1stResult}. For the expected value of the length of the first monotone run, one can then apply Lebesgue's dominated convergence theorem. However, in the present article we aim to avoid restrictive assumptions on $\mathbf{F}$ and instead focus on an ergodic version of the result.

Under Assumption \ref{Assumptions}, we say that the SDE \eqref{eqGenSDEIntro} admits an invariant probability measure $\mu$ on $\R^d$, if the initial condition $\mathbf{X}_0 \sim \mu$ implies that the corresponding solution process $\mathbf{X} = (\mathbf{X}_t)_{t \geq 0}$ is stationary. In addition to assuming that SDE \eqref{eqGenSDEIntro} admits an invariant probability measure $\mu$, our second main result requires that $\mathbf{X}_0 \sim \mu$, and that each skeleton Markov chain $(\mathbf{X}_{n \delta})_{n \in \N_0}$, $\delta > 0$, is ergodic. The ergodicity assumption is equivalent to the assumption that
\begin{equation}
\label{eqErgoCond}
\lim_{N \rightarrow \infty} \frac{1}{N} \sum_{n=0}^{N-1} \mathbbm{1}_B(\mathbf{X}_{n\delta}) = \mu(B) \quad \text{in probability} \quad \text{for all } B \in \mathcal{B}(\R^d), \, \delta > 0,
\end{equation} 
see Lemma~\ref{LemmaErgodicity} in Appendix~\ref{SecAppendixErgodicity} for details. Furthermore, a sufficient condition for the ergodicity of the skeleton Markov chain $(\mathbf{X}_{n \delta})_{n \in \N_0}$ is its irreducibility, compare Lemma~\ref{LemmaIrreducibility}.

\begin{theorem}[Mean length of monotone runs in a single path]
\label{Thm2ndResult}
Suppose that Assumption~\ref{Assumptions} holds and that \eqref{eqGenSDEIntro} admits an invariant measure $\mu$. Let $\mathbf{X}_0 \sim \mu$ be independent of $\mathbf{W}$, let $\mathbf{X} = (\mathbf{X}_t)_{t \geq 0}$ be the corresponding solution process to \eqref{eqGenSDEIntro} with $d$-th component $X = (X_t)_{t \geq 0}$, and assume that the ergodicity condition \eqref{eqErgoCond} is fulfilled. Then it holds that
\begin{equation}
\label{eqStatementMainResult2}
\lim_{\delta \rightarrow 0} \lim_{J \rightarrow \infty} \frac{1}{J} \sum_{j=1}^J T_{j,\delta}(X) = 2 \quad P\text{-almost surely}.
\end{equation}
Furthermore, for $\mu$-almost all $\mathbf{x} \in \R^d$, \eqref{eqStatementMainResult2} also holds for the $d$-th component $X^{\mathbf{x}} = (X_t^{\mathbf{x}})_{t \geq 0}$ of the solution process $\mathbf{X}^{\mathbf{x}} = (\mathbf{X}_t^{\mathbf{x}})_{t \geq 0}$ with initial condition $\mathbf{X}_0^{\mathbf{x}} = \mathbf{x}$ $P$-a.s., in place of $X$.
\end{theorem}

\begin{remark}[Remaining monotone run length]
\label{RemarkRemainingMonLength}
An alternative approximation of the expected value of the length of the first monotone run is given by the mean of the \textnormal{remaining} monotone run lengths. That is, instead of counting the numbers of consecutive increments having the same sign, one counts the number of remaining increments until the next change of signs appears starting from every single data point (e.g., counting a monotone run length of 3 is equivalent to adding up remaining run lengths of 3, 2 and 1). More formally, we define the $n$-th remaining run length by
\begin{equation*}
\varphi(\mathbf{X}_{(n-1)\delta}, \mathbf{X}_{n \delta}, \ldots) \coloneq \min \lbrace \ell \in \N_{\geq n}: \Delta X_{\ell+1,\delta} \Delta X_{\ell,\delta} < 0 \rbrace,
\end{equation*}
which in the setting of ergodic Markov chains $(\mathbf{X}_{n \delta})_{n \in \N_0}$ almost surely takes values in $\N$, see Lemma \ref{LemmaFiniteOccurrenceTimes}.
Then the mapping $\varphi: (\R^d)^\N \to \N_0$ is measurable and a direct consequence of the ergodicity is
\begin{equation*}
\lim_{N \rightarrow \infty} \frac{1}{N} \sum_{n=1}^N \varphi(\mathbf{X}_{(n-1)\delta}, \mathbf{X}_{n \delta}, \ldots) = \text{E}(\varphi(\mathbf{X}_{0}, \mathbf{X}_{\delta}, \ldots)) = \text{E}(T_{1,\delta}(X))
\end{equation*}
for all $\delta > 0$, see, e.g., \cite[Proposition 4.3]{Krengel1985}. Since the length of monotone runs is the more intuitive quantity and the approach of remaining monotone run lengths has not shown any advantages in numerical tests, we will continue to use the former perspective.
\end{remark}

\subsection{Derivation of the analytical results}
\label{SubsectionProofs}

To prove our main results, we begin with a lemma that gives lower and upper bounds for  probabilities of the form $P(T_{1,\delta}(X) = \ell, \sup_{t \in [0, (\ell+1) \delta]} \Vert \mathbf{X}_t \Vert \geq R)$ and $P(T_{1,\delta}(X) \geq \ell, \sup_{t \in [0, \ell \delta]} \Vert \mathbf{X}_t \Vert \geq R)$, where $\delta, R > 0$, for solution processes $\mathbf{X}$ in the sense of Assumption~\ref{Assumptions}. These bounds do not require $F$ to satisfy a polynomial growth condition. However, this assumption is used in combination with the bounds in the proof of Theorem \ref{Thm1stResult}.

For the proof of our second main result, we first show that, in the case of ergodicity, almost surely there are infinitely many occurrence times $R_{j, \delta}(X)$, or equivalently, $(X_{n \delta})_{n \in \N_0}$ has infinitely many monotone runs, for $\delta > 0$ sufficiently small. Building on this, we can finally prove Theorem \ref{Thm2ndResult}.

\begin{lemma}
\label{Lemma1stResult}
Suppose that Assumption \ref{Assumptions} holds, let $\mathbf{X}_0$ be an initial random variable independent of $\mathbf{W}$, and let $\mathbf{X} = (\mathbf{X}_t)_{t \geq 0}$ be the corresponding solution to \eqref{eqGenSDEIntro} with $d$-th component $X = (X_t)_{t \geq 0}$. Then for all $\ell \in \N$, all $R > 0$ and all $\delta > 0$ it holds that
\begin{equation}
\label{eqLemma1stResult01}
\begin{aligned}
&2  \Phi\Bigl(- \tfrac{\sqrt{\delta}}{A} \sup_{\Vert \mathbf{x} \Vert \leq R} \vert F(\mathbf{x}) \vert \Bigr)^{\ell+1} - 2 \sum_{j = 1}^\ell \Phi\Bigl(-\tfrac{\sqrt{\delta}}{A} \sup_{\Vert \mathbf{x} \Vert \leq R} \vert F(\mathbf{x}) \vert\Bigr)^j P \Bigl(\sup_{t \in I_{\ell-j}} \Vert \mathbf{X}_t \Vert > R \Bigr) \\
&\quad \leq P \Bigl(T_{1,\delta}(X) = \ell, \sup_{t \in [0, (\ell+1) \delta]} \Vert \mathbf{X}_t \Vert \leq R \Bigr) \\
&\quad \leq 2 \Phi\Bigl(\tfrac{\sqrt{\delta}}{A} \sup_{\Vert \mathbf{x} \Vert \leq R} \vert F(\mathbf{x}) \vert\Bigr)^{\ell+1}
\end{aligned}
\end{equation}
and
\begin{equation}
\label{eqLemma1stResult02}
\begin{aligned}
&2 \Phi\Bigl(-\tfrac{\sqrt{\delta}}{A} \sup_{\Vert \mathbf{x} \Vert \leq R} \vert F(\mathbf{x}) \vert\Bigr)^\ell - 2 \sum_{j = 1}^{\ell-1} \Phi\Bigl(-\tfrac{\sqrt{\delta}}{A} \sup_{\Vert \mathbf{x} \Vert \leq R} \vert F(\mathbf{x}) \vert\Bigr)^j P \Bigl(\sup_{t \in I_{\ell-j}} \Vert \mathbf{X}_t \Vert > R \Bigr) \\
&\quad \leq P \Bigl(T_{1,\delta}(X) \geq \ell, \sup_{t \in [0, \ell \delta]} \Vert \mathbf{X}_t \Vert \leq R \Bigr) \\
&\quad \leq 2 \Phi\Bigl(\tfrac{\sqrt{\delta}}{A} \sup_{\Vert \mathbf{x} \Vert \leq R} \vert F(\mathbf{x}) \vert\Bigr)^\ell,
\end{aligned}
\end{equation}
where $\Phi$ denotes the cumulative distribution function of the standard normal distribution and $I_1 \coloneq [0, \delta]$, $I_j \coloneq ((j-1) \delta, j \delta]$, $j = 2, \ldots, \ell+1$.
\end{lemma}

\begin{proof}
Let $\ell \in \N$, $R > 0$ and $\delta > 0$ be fixed.\\
We begin by proving \eqref{eqLemma1stResult01}. First note that
\begin{align*}
&\Bigl\lbrace T_{1}(X) = \ell, \sup_{t \in [0, (\ell+1) \delta]} \Vert \mathbf{X}_t \Vert \leq R \Bigr\rbrace\\
&\quad = \Bigl\lbrace \Delta X_{1} \leq 0, \Delta X_{2} \leq 0, \ldots, \Delta X_\ell \leq 0, \Delta X_{\ell+1} > 0, \sup_{t \in [0, (\ell+1) \delta]} \Vert \mathbf{X}_t \Vert \leq R \Bigr\rbrace \\
&\quad \quad \; \cup \Bigl\lbrace \Delta X_{1} \geq 0, \Delta X_{2} \geq 0, \ldots, \Delta X_\ell \geq 0, \Delta X_{\ell+1} < 0, \sup_{t \in [0, (\ell+1) \delta]} \Vert \mathbf{X}_t \Vert \leq R \Bigr\rbrace.
\end{align*}
Define $\mathcal{F}_t \coloneq \sigma((\mathbf{X}_s)_{s\leq t})$, $t \geq 0$. Then
\begin{align*}
&P\Bigl( \Delta X_{1} \leq 0, \Delta X_{2} \leq 0, \ldots, \Delta X_\ell \leq 0, \Delta X_{\ell+1} > 0, \sup_{t \in [0, (\ell+1) \delta]} \Vert \mathbf{X}_t \Vert \leq R \Bigr) \\
&= \text{E} \Bigl[ \mathbbm{1}_{\lbrace \Delta X_{1} \leq 0, \sup_{t \in I_1} \Vert \mathbf{X}_t \Vert \leq R \rbrace} \mathbbm{1}_{\lbrace \Delta X_{2} \leq 0, \sup_{t \in I_2} \Vert \mathbf{X}_t \Vert \leq R}\rbrace \ldots \\
&\hspace{2.5em}\mathbbm{1}_{\lbrace \Delta X_\ell \leq 0, \sup_{t \in I_\ell} \Vert \mathbf{X}_t \Vert \leq R \rbrace} \mathbbm{1}_{\lbrace \Delta X_{\ell+1} > 0, \sup_{t \in I_{\ell+1}} \Vert \mathbf{X}_t \Vert \leq R \rbrace} \Bigr] \\
&= \text{E} \Bigl[ \text{E} \bigl[ \ldots \text{E} \bigl[ \mathbbm{1}_{\lbrace \Delta X_{1} \leq 0, \sup_{t \in I_1} \Vert \mathbf{X}_t \Vert \leq R \rbrace} \mathbbm{1}_{\lbrace \Delta X_{2} \leq 0, \sup_{t \in I_2} \Vert \mathbf{X}_t \Vert \leq R \rbrace} \ldots\\
&\hspace{6em} \mathbbm{1}_{\lbrace \Delta X_\ell \leq 0, \sup_{t \in I_\ell} \Vert \mathbf{X}_t \Vert \leq R \rbrace} \mathbbm{1}_{\lbrace \Delta X_{\ell+1} > 0, \sup_{t \in I_{\ell+1}} \Vert \mathbf{X}_t \Vert \leq R \rbrace} \, \vert \, \mathcal{F}_{\ell \delta} \bigr] \ldots \, \vert \, \mathcal{F}_{\delta} \bigr] \Bigr]
\end{align*}
by the tower property of conditional expectations. By definition of $(\mathcal{F}_t)_{t \geq 0}$, the indicator functions $\mathbbm{1}_{\lbrace \Delta X_i \leq 0, \sup_{t \in I_i} \Vert \mathbf{X}_t \Vert \leq R \rbrace}$ (or $\mathbbm{1}_{\lbrace \Delta X_i > 0, \sup_{t \in I_i} \Vert \mathbf{X}_t \Vert \leq R \rbrace}$ respectively), $1 \leq i \leq j \leq \ell+1$, are $\mathcal{F}_{j \delta}$-measurable, so by the properties of conditional expectations,
\begin{align*}
&P\Bigl( \Delta X_{1} \leq 0, \Delta X_{2} \leq 0, \ldots, \Delta X_\ell \leq 0, \Delta X_{\ell+1} > 0, \sup_{t \in [0, (\ell+1) \delta]} \Vert \mathbf{X}_t \Vert \leq R \Bigr) \\
&= \text{E} \Bigl[\mathbbm{1}_{\lbrace \Delta X_{1} \leq 0, \sup_{t \in I_1} \Vert \mathbf{X}_t \Vert \leq R \rbrace} \text{E} \bigl[ \mathbbm{1}_{\lbrace \Delta X_{2} \leq 0, \sup_{t \in I_2} \Vert \mathbf{X}_t \Vert \leq R \rbrace} \ldots\\
&\hspace{2.5em}\text{E} \bigl[ \mathbbm{1}_{\lbrace \Delta X_\ell \leq 0, \sup_{t \in I_\ell} \Vert \mathbf{X}_t \Vert \leq R \rbrace} \text{E} \bigl[ \mathbbm{1}_{\lbrace \Delta X_{\ell+1} > 0, \sup_{t \in I_{\ell+1}} \Vert \mathbf{X}_t \Vert \leq R \rbrace} \, \vert \, \mathcal{F}_{\ell \delta} \bigr] \, \vert \, \mathcal{F}_{(\ell-1) \delta} \bigr] \ldots \, \vert \, \mathcal{F}_{\delta} \bigr] \Bigr].
\end{align*}
The uniqueness formulated for $\mathbf{X}$ in Assumption \ref{Assumptions} implies that $\mathbf{X}_s = \mathbf{X}_s^{\mathbf{X}_{\ell\delta}, \ell \delta}$ a.s.\ for all $s \geq \ell \delta$, so for the conditional expectation at the very inside we get
\begin{align*}
&\text{E} \bigl[ \mathbbm{1}_{\lbrace \Delta X_{\ell+1} > 0, \sup_{t \in I_{\ell+1}} \Vert \mathbf{X}_t \Vert \leq R \rbrace} \, \vert \, \mathcal{F}_{\ell \delta} \bigr]\\
&\quad = \,\text{E} \biggl[ \mathbbm{1}_{\lbrace X_{(\ell+1)\delta}^{\mathbf{X}_{\ell\delta}, \ell \delta}-X_{\ell \delta}^{\mathbf{X}_{\ell\delta}, \ell \delta} > 0, \sup_{t \in I_{\ell+1}} \Vert \mathbf{X}^{\mathbf{X}_{\ell\delta}, \ell \delta}_t \Vert \leq R \rbrace} \, \vert \, \mathcal{F}_{\ell \delta} \biggr]\\
&\quad = \,\text{E} \biggl[ \mathbbm{1}_{\lbrace X_{(\ell+1)\delta}^{\mathbf{X}_{\ell\delta}, \ell \delta}-X_{\ell \delta}^{\mathbf{X}_{\ell\delta}, \ell \delta} > 0, \sup_{t \in I_{\ell+1}} \Vert \mathbf{X}^{\mathbf{X}_{\ell\delta}, \ell \delta}_t \Vert \leq R \rbrace} \, \vert \, \mathbf{X}_{\ell \delta} \biggr],
\end{align*}
where the last line results from the Markov property since $\mathbf{X}$ is a time-homogeneous Itô diffusion. Using, e.g., \cite[Lemma 2.3.4]{Shreve2004} it follows that
\begin{align*}
&\text{E} \biggl[ \mathbbm{1}_{\lbrace X_{(\ell+1)\delta}^{\mathbf{X}_{\ell\delta}, \ell \delta}-X_{\ell \delta}^{\mathbf{X}_{\ell\delta}, \ell \delta} > 0, \sup_{t \in I_{\ell+1}} \Vert \mathbf{X}^{\mathbf{X}_{\ell\delta}, \ell \delta}_t \Vert \leq R \rbrace} \, \vert \, \mathbf{X}_{\ell \delta} \biggr]\\
&\quad = \, \text{E} \biggl[ \mathbbm{1}_{\lbrace X_{(\ell+1)\delta}^{\mathbf{x}, \ell \delta}-X_{\ell \delta}^{\mathbf{x}, \ell \delta} > 0, \sup_{t \in I_{\ell+1}} \Vert \mathbf{X}^{\mathbf{x}, \ell \delta}_t \Vert \leq R \rbrace} \biggr]\biggl\vert_{\mathbf{x} = \mathbf{X}_{\ell \delta}}\\
&\quad =\, P \Bigl( X_{(\ell+1)\delta}^{\mathbf{x}, \ell \delta}-X_{\ell \delta}^{\mathbf{x}, \ell \delta} > 0, \sup_{t \in I_{\ell+1}} \Vert \mathbf{X}^{\mathbf{x}, \ell \delta}_t \Vert \leq R  \Bigr)\biggl\vert_{\mathbf{x} = \mathbf{X}_{\ell \delta}}.
\end{align*}
Let $\Delta W_{\ell+1} \coloneq W_{(\ell+1)\delta} - W_{\ell\delta}$. Then, by definition of $X$,
\begin{align*}
&P \Bigl( X_{(\ell+1)\delta}^{\mathbf{x}, \ell \delta}-X_{\ell \delta}^{\mathbf{x}, \ell \delta} > 0, \sup_{t \in I_{\ell+1}} \Vert \mathbf{X}^{\mathbf{x}, \ell \delta}_t \Vert \leq R  \Bigr)\biggl\vert_{\mathbf{x} = \mathbf{X}_{\ell \delta}}\\
&\quad = \, P\Bigl( \int_{\ell \delta}^{(\ell+1)\delta} F(\mathbf{X}_s^{\mathbf{x}, \ell\delta}) \, \text{d}s + A \Delta W_{\ell+1} > 0, \, \sup_{t \in I_{\ell+1}} \Vert \mathbf{X}_t^{\mathbf{x}, \ell\delta} \Vert \leq R \Bigr)\biggl\vert_{\mathbf{x} = \mathbf{X}_{\ell \delta}}\\
&\quad= \,P\Bigl( A \tfrac{\Delta W_{\ell+1}}{\sqrt{\delta}} > -\tfrac{\int_{\ell\delta}^{(\ell+1)\delta} F(\mathbf{X}_s^{\mathbf{x},\ell\delta}) \, \text{d}s}{\sqrt{\delta}}, \, \sup_{t \in I_{\ell+1}} \Vert \mathbf{X}_t^{\mathbf{x}, \ell\delta} \Vert \leq R \Bigr) \biggl\vert_{\mathbf{x} = \mathbf{X}_{\ell \delta}}.
\end{align*}
Furthermore,
\begin{align*}
&P\Bigl( A \tfrac{\Delta W_{\ell+1}}{\sqrt{\delta}} > -\tfrac{\int_{\ell\delta}^{(\ell+1)\delta} F(\mathbf{X}_s^{\mathbf{x},\ell\delta}) \, \text{d}s}{\sqrt{\delta}}, \, \sup_{t \in I_{\ell+1}} \Vert \mathbf{X}_t^{\mathbf{x}, \ell\delta} \Vert \leq R \Bigr) \biggl\vert_{\mathbf{x} = \mathbf{X}_{\ell \delta}} \\
&\quad \geq \, P\Bigl( A \tfrac{\Delta W_{\ell+1}}{\sqrt{\delta}} >\sqrt{\delta} \sup_{\Vert \mathbf{x} \Vert \leq R} \vert F(\mathbf{x}) \vert , \, \sup_{t \in I_{\ell+1}} \Vert \mathbf{X}_t^{\mathbf{x}, \ell\delta} \Vert \leq R \Bigr)\biggl\vert_{\mathbf{x} = \mathbf{X}_{\ell\delta}}\\
&\quad \geq \, P\Bigl( A \tfrac{\Delta W_{\ell+1}}{\sqrt{\delta}} >\sqrt{\delta} \sup_{\Vert \mathbf{x} \Vert \leq R} \vert F(\mathbf{x}) \vert \Bigr) - P\Bigl(\sup_{t \in I_{\ell+1}} \Vert \mathbf{X}_t^{\mathbf{x}, \ell\delta} \Vert > R \Bigr)\biggl\vert_{\mathbf{x} = \mathbf{X}_{\ell \delta}} \\
&\quad = \, \Phi \biggl(-\tfrac{\sqrt{\delta}}{A} \sup_{\Vert \mathbf{x} \Vert \leq R} \vert F(\mathbf{x}) \vert \biggr) - \text{E}\Bigl[\mathbbm{1}_{ \lbrace \sup_{t \in I_{\ell+1}} \Vert \mathbf{X}_t \Vert > R \rbrace} \, \vert \, \mathcal{F}_{\ell \delta} \Bigr]
\end{align*}
and
\begin{align*}
&P\Bigl( A \tfrac{\Delta W_{\ell+1}}{\sqrt{\delta}} > -\tfrac{\int_{\ell\delta}^{(\ell+1)\delta} F(\mathbf{X}_s^{\mathbf{x},\ell\delta}) \, \text{d}s}{\sqrt{\delta}}, \sup_{t \in I_{\ell+1}} \Vert \mathbf{X}_t^{\mathbf{x}, \ell\delta} \Vert \leq R \Bigr) \biggl\vert_{\mathbf{x} = \mathbf{X}_{\ell \delta}} \\
&\quad \leq \, P\Bigl( A \tfrac{\Delta W_{\ell+1}}{\sqrt{\delta}} > -\sqrt{\delta}\sup_{\Vert \mathbf{x} \Vert \leq R} \vert F(\mathbf{x}) \vert \Bigr)\\ 
&\quad = \, \Phi \biggl(\tfrac{\sqrt{\delta}}{A} \sup_{\Vert \mathbf{x} \Vert \leq R} \vert F(\mathbf{x}) \vert \biggr),
\end{align*}
where we have used that $\delta^{-1/2}\Delta W_{\ell+1}$ is standard normally distributed and that the corresponding cumulative distribution function $\Phi$ has the property of $\Phi(-x) = 1-\Phi(x)$ for $x \in \R$. Therefore, a lower bound for the full iterated conditional expectations is
\begin{align*}
& \Phi\biggl(-\tfrac{\sqrt{\delta}}{A} \sup_{\Vert \mathbf{x} \Vert \leq R} \vert F(\mathbf{x}) \vert \biggr) \text{E} \Bigl[\mathbbm{1}_{\lbrace \Delta X_{1} \leq 0, \sup_{t \in I_1} \Vert \mathbf{X}_t \Vert \leq R \rbrace} \text{E} \bigl[ \mathbbm{1}_{\lbrace \Delta X_{2} \leq 0, \sup_{t \in I_2} \Vert \mathbf{X}_t \Vert \leq R \rbrace}\\
&\hspace{12em} \ldots \text{E} \bigl[ \mathbbm{1}_{\lbrace \Delta X_\ell \leq 0, \sup_{t \in I_\ell} \Vert \mathbf{X}_t \Vert \leq R \rbrace} \, \vert \, \mathcal{F}_{(\ell-1) \delta} \bigr] \ldots \, \vert \, \mathcal{F}_{\delta} \bigr] \Bigr] \\
&-\text{E} \Bigl[\mathbbm{1}_{\lbrace \Delta X_{1} \leq 0, \sup_{t \in I_1} \Vert \mathbf{X}_t \Vert \leq R \rbrace} \text{E} \bigl[ \mathbbm{1}_{\lbrace \Delta X_{2} \leq 0, \sup_{t \in I_2} \Vert \mathbf{X}_t \Vert \leq R \rbrace}\\
&\hspace{1.5em} \ldots \text{E} \bigl[ \mathbbm{1}_{\lbrace \Delta X_\ell \leq 0, \sup_{t \in I_\ell} \Vert \mathbf{X}_t \Vert \leq R \rbrace} \text{E}\Bigl[\mathbbm{1}_{ \lbrace \sup_{t \in I_{\ell+1}} \Vert \mathbf{X}_t \Vert > R \rbrace} \, \vert \, \mathcal{F}_{\ell \delta} \Bigr] \, \vert \, \mathcal{F}_{(\ell-1) \delta} \bigr] \ldots \, \vert \, \mathcal{F}_{\delta} \bigr] \Bigr].
\end{align*}
Since, due to the monotonicity and the tower property of conditional expectations,
\begin{align*}
&\text{E} \biggl[ \mathbbm{1}_{\lbrace \Delta X_{j} \leq 0, \sup_{t \in I_j} \Vert \mathbf{X}_t \Vert \leq R \rbrace} \text{E}\Bigl[\mathbbm{1}_{ \lbrace \sup_{t \in I_{\ell+1}} \Vert \mathbf{X}_t \Vert > R \rbrace} \, \vert \, \mathcal{F}_{j \delta} \Bigr] \, \vert \, \mathcal{F}_{(j-1) \delta} \biggr] \\ &\quad \leq \, \text{E} \biggl[ \text{E}\Bigl[\mathbbm{1}_{ \lbrace \sup_{t \in I_{\ell+1}} \Vert \mathbf{X}_t \Vert > R \rbrace} \, \vert \, \mathcal{F}_{j \delta} \Bigr] \, \vert \, \mathcal{F}_{(j-1) \delta} \biggr] \\
&\quad = \, \text{E} \Bigl[ \mathbbm{1}_{ \lbrace \sup_{t \in I_{\ell+1}} \Vert \mathbf{X}_t \Vert > R \rbrace}  \, \vert \, \mathcal{F}_{(j-1) \delta} \Bigr] 
\end{align*}
for $j = 2, \ldots, \ell$, in particular
\begin{align*}
&\Phi\biggl(-\tfrac{\sqrt{\delta}}{A} \sup_{\Vert \mathbf{x} \Vert \leq R} \vert F(\mathbf{x}) \vert \biggr) \text{E} \Bigl[\mathbbm{1}_{\lbrace \Delta X_{1} \leq 0, \sup_{t \in I_1} \Vert \mathbf{X}_t \Vert \leq R \rbrace} \text{E} \bigl[ \mathbbm{1}_{\lbrace \Delta X_{2} \leq 0, \sup_{t \in I_2} \Vert \mathbf{X}_t \Vert \leq R \rbrace}\\
&\hspace{12em} \ldots \text{E} \bigl[ \mathbbm{1}_{\lbrace \Delta X_\ell \leq 0, \sup_{t \in I_\ell} \Vert \mathbf{X}_t \Vert \leq R \rbrace} \, \vert \, \mathcal{F}_{(\ell-1) \delta} \bigr] \ldots \, \vert \, \mathcal{F}_{\delta} \bigr] \Bigr] \\
& -\text{E} \Bigl[ \text{E}\Bigl[\mathbbm{1}_{ \lbrace \sup_{t \in I_{\ell+1}} \Vert \mathbf{X}_t \Vert > R \rbrace} \, \vert \, \mathcal{F}_{\delta} \Bigr] \Bigr] \\
= \, &\Phi\biggl(-\tfrac{\sqrt{\delta}}{A} \sup_{\Vert \mathbf{x} \Vert \leq R} \vert F(\mathbf{x}) \vert \biggr) \text{E} \Bigl[\mathbbm{1}_{\lbrace \Delta X_{1} \leq 0, \sup_{t \in I_1} \Vert \mathbf{X}_t \Vert \leq R \rbrace} \text{E} \bigl[ \mathbbm{1}_{\lbrace \Delta X_{2} \leq 0, \sup_{t \in I_2} \Vert \mathbf{X}_t \Vert \leq R \rbrace}\\
&\hspace{12em} \ldots \text{E} \bigl[ \mathbbm{1}_{\lbrace \Delta X_\ell \leq 0, \sup_{t \in I_\ell} \Vert \mathbf{X}_t \Vert \leq R \rbrace} \, \vert \, \mathcal{F}_{(\ell-1) \delta} \bigr] \ldots \, \vert \, \mathcal{F}_{\delta} \bigr] \Bigr] \\
&- P \Bigl(\sup_{t \in I_{\ell+1}} \Vert \mathbf{X}_t \Vert > R \Bigr)
\end{align*}
is a lower bound for the full iterated conditional expectations.
Altogether we thus have
\begin{align*}
& \Phi\biggl(-\tfrac{\sqrt{\delta}}{A} \sup_{\Vert \mathbf{x} \Vert \leq R} \vert F(\mathbf{x}) \vert \biggr) \text{E} \Bigl[\mathbbm{1}_{\lbrace \Delta X_{1} \leq 0, \sup_{t \in I_1} \Vert \mathbf{X}_t \Vert \leq R \rbrace} \text{E} \bigl[ \mathbbm{1}_{\lbrace \Delta X_{2} \leq 0, \sup_{t \in I_2} \Vert \mathbf{X}_t \Vert \leq R \rbrace}\\
&\hspace{12em} \ldots \text{E} \bigl[ \mathbbm{1}_{\lbrace \Delta X_\ell \leq 0, \sup_{t \in I_\ell} \Vert \mathbf{X}_t \Vert \leq R \rbrace} \, \vert \, \mathcal{F}_{(\ell-1) \delta} \bigr] \ldots \, \vert \, \mathcal{F}_{\delta} \bigr] \Bigr] \\
&- P \Bigl(\sup_{t \in I_{\ell+1}} \Vert \mathbf{X}_t \Vert > R \Bigr) \\
\leq \, &\text{E} \Bigl[\mathbbm{1}_{\lbrace \Delta X_{1} \leq 0, \sup_{t \in I_1} \Vert \mathbf{X}_t \Vert \leq R \rbrace} \text{E} \bigl[ \mathbbm{1}_{\lbrace \Delta X_{2} \leq 0, \sup_{t \in I_2} \Vert \mathbf{X}_t \Vert \leq R \rbrace}\\
&\hspace{1em} \ldots \text{E} \bigl[ \mathbbm{1}_{\lbrace \Delta X_\ell \leq 0, \sup_{t \in I_\ell} \Vert \mathbf{X}_t \Vert \leq R \rbrace} \text{E} \bigl[ \mathbbm{1}_{\lbrace \Delta X_{\ell+1} > 0, \sup_{t \in I_{\ell+1}} \Vert \mathbf{X}_t \Vert \leq R \rbrace} \, \vert \, \mathcal{F}_{\ell \delta} \bigr] \, \vert \, \mathcal{F}_{\ell-1 \delta} \bigr] \ldots \, \vert \, \mathcal{F}_{\delta} \bigr] \Bigr] \\
\leq \, & \Phi \biggl(\tfrac{\sqrt{\delta}}{A} \sup_{\Vert \mathbf{x} \Vert \leq R} \vert F(\mathbf{x}) \vert \biggr) \text{E} \Bigl[\mathbbm{1}_{\lbrace \Delta X_{1} \leq 0, \sup_{t \in I_1} \Vert \mathbf{X}_t \Vert \leq R \rbrace} \text{E} \bigl[ \mathbbm{1}_{\lbrace \Delta X_{2} \leq 0, \sup_{t \in I_2} \Vert \mathbf{X}_t \Vert \leq R \rbrace}\\
&\hspace{11em} \ldots \text{E} \bigl[ \mathbbm{1}_{\lbrace \Delta X_\ell \leq 0, \sup_{t \in I_\ell} \Vert \mathbf{X}_t \Vert \leq R \rbrace} \, \vert \, \mathcal{F}_{(\ell-1) \delta} \bigr] \ldots \, \vert \, \mathcal{F}_{\delta} \bigr] \Bigr].
\end{align*}
We can proceed analogously for the next conditional expectation in the very inside of the lower and upper bounds (note that replacing $>$ by $\leq$ in the indicator function leads to the same factors $\Phi(\pm \sqrt{\delta} A^{-1} \sup_{\Vert \mathbf{x} \Vert \leq R} \vert F(\mathbf{x}) \vert)$). Finally, we have
\begin{align*}
&\Phi\biggl(-\tfrac{\sqrt{\delta}}{A} \sup_{\Vert \mathbf{x} \Vert \leq R} \vert F(\mathbf{x}) \vert\biggr)^{\ell+1} - \sum_{j = 1}^\ell \Phi\biggl(-\tfrac{\sqrt{\delta}}{A} \sup_{\Vert \mathbf{x} \Vert \leq R} \vert F(\mathbf{x}) \vert\biggr)^j P \Bigl(\sup_{t \in I_{\ell-j}} \Vert \mathbf{X}_t \Vert > R \Bigr) \\
&\quad \leq \, P \Bigl( \Delta X_{1} \leq 0, \Delta X_{2} \leq 0, \ldots, \Delta X_\ell \leq 0, \Delta X_{\ell+1} > 0, \sup_{t \in [0, (\ell+1) \delta]} \Vert \mathbf{X}_t \Vert \leq R \Bigr)\\
&\quad \leq \, \Phi\biggl(\tfrac{\sqrt{\delta}}{A} \sup_{\Vert \mathbf{x} \Vert \leq R} \vert F(\mathbf{x}) \vert\biggr)^{\ell+1}.
\end{align*}
By similar arguments,
\begin{align*}
&\Phi\biggl(-\tfrac{\sqrt{\delta}}{A} \sup_{\Vert \mathbf{x} \Vert \leq R} \vert F(\mathbf{x}) \vert\biggr)^{\ell+1} - \sum_{j = 1}^\ell \Phi\biggl(-\tfrac{\sqrt{\delta}}{A} \sup_{\Vert \mathbf{x} \Vert \leq R} \vert F(\mathbf{x}) \vert\biggr)^j P \Bigl(\sup_{t \in I_{\ell-j}} \Vert \mathbf{X}_t \Vert > R \Bigr) \\
&\quad \leq \, P \Bigl( \Delta X_{1} \geq 0, \Delta X_{2} \geq 0, \ldots, \Delta X_\ell \geq 0, \Delta X_{\ell+1} < 0, \sup_{t \in [0, (\ell+1) \delta]} \Vert \mathbf{X}_t \Vert \leq R \Bigr)\\
&\quad \leq \, \Phi\biggl(\tfrac{\sqrt{\delta}}{A} \sup_{\Vert \mathbf{x} \Vert \leq R} \vert F(\mathbf{x}) \vert\biggr)^{\ell+1}.
\end{align*}
Altogether,
\begin{align*}
&2 \Phi\biggl(-\tfrac{\sqrt{\delta}}{A} \sup_{\Vert \mathbf{x} \Vert \leq R} \vert F(\mathbf{x}) \vert\biggr)^{\ell+1} - 2 \sum_{j = 1}^\ell \Phi\biggl(-\tfrac{\sqrt{\delta}}{A} \sup_{\Vert \mathbf{x} \Vert \leq R} \vert F(\mathbf{x}) \vert\biggr)^j P \Bigl(\sup_{t \in I_{\ell-j}} \Vert \mathbf{X}_t \Vert > R \Bigr) \\
&\quad \leq \, P \Bigl(T_1 = \ell, \sup_{t \in [0, (\ell+1) \delta]} \Vert \mathbf{X}_t \Vert \leq R \Bigr) \\
&\quad \leq \, 2 \Phi\biggl(\tfrac{\sqrt{\delta}}{A} \sup_{\Vert \mathbf{x} \Vert \leq R} \vert F(\mathbf{x}) \vert\biggr)^{\ell+1}.
\end{align*}
We now turn to the proof of \eqref{eqLemma1stResult02}. Since
\begin{align*}
&P \Bigl(T_{1} \geq \ell, \sup_{t \in [0, \ell \delta]} \Vert \mathbf{X}_t \Vert \leq R \Bigr) \\
&\quad = \, P \Bigl( \Delta X_{1} \leq 0, \Delta X_{2} \leq 0, \ldots, \Delta X_\ell \leq 0, \sup_{t \in [0, \ell \delta]} \Vert \mathbf{X}_t \Vert \leq R \Bigr) \\
&\quad \quad \; + P \Bigl( \Delta X_{1} \geq 0, \Delta X_{2} \geq 0, \ldots, \Delta X_\ell > 0, \sup_{t \in [0, \ell \delta]} \Vert \mathbf{X}_t \Vert \leq R \Bigr),
\end{align*}
similar arguments as used for the previous claim show directly that
\begin{align*}
&\Phi\Bigl(-\sqrt{\delta}\tfrac{\sup_{\Vert \mathbf{x} \Vert \leq R} \vert F(\mathbf{x}) \vert}{A}\Bigr)^\ell - 2 \sum_{j = 1}^{\ell-1} \Phi\Bigl(-\sqrt{\delta}\tfrac{\sup_{\Vert \mathbf{x} \Vert \leq R} \vert F(\mathbf{x}) \vert}{A}\Bigr)^j P \Bigl(\sup_{t \in I_{\ell-j}} \Vert \mathbf{X}_t \Vert > R \Bigr)\\
&\quad \leq \, P \Bigl(T_1 \geq \ell, \sup_{t \in [0, \ell \delta]} \Vert \mathbf{X}_t \Vert \leq R \Bigr) \\
&\quad \leq \, 2 \Phi\Bigl(\sqrt{\delta}\tfrac{\sup_{\Vert \mathbf{x} \Vert \leq R} \vert F(\mathbf{x}) \vert}{A}\Bigr)^\ell.
\end{align*}
\end{proof}

Using the polynomial growth condition for $F$, we can now prove Theorem \ref{Thm1stResult}.

\begin{proof}[Proof of Theorem \ref{Thm1stResult}]
We start the proof by showing the claim for finite lengths of monotone runs, i.e., $\ell \in \N$. For fixed $\delta >0$, $\ell \in \N$, it holds that
\begin{align*}
P(T_{1, \delta}(X) = \ell) = \, &P \Bigl(T_{1, \delta}(X) = \ell, \sup_{t \in [0, (\ell+1) \delta]} \Vert \mathbf{X}_t \Vert > R \Bigr) \\
&+ P \Bigl(T_{1, \delta}(X) = \ell, \sup_{t \in [0, (\ell+1) \delta]} \Vert \mathbf{X}_t \Vert \leq R \Bigr).
\end{align*}
For the limit of the first summand note that
\begin{align*}
0 \leq \lim_{\delta \rightarrow 0} P \Bigl(T_{1, \delta}(X) = \ell, \sup_{t \in [0, (\ell+1) \delta]} \Vert \mathbf{X}_t \Vert > R \Bigr) &\leq \lim_{\delta \rightarrow 0} P \Bigl(\sup_{t \in [0, (\ell+1) \delta]} \Vert \mathbf{X}_t \Vert > R \Bigr) \\
&\leq \lim_{\delta \rightarrow 0} P \Bigl(\sup_{t \in [0, 1]} \Vert \mathbf{X}_t \Vert > \delta^{-\alpha} \Bigr) = 0,
\end{align*}
due to the continuous paths of $\mathbf{X}$ being bounded on the compact set $[0,1]$.
For the limit of the second summand let $\alpha \in (0, (2p)^{-1})$ and set $R = R_{\delta} = \delta^{-\alpha}$. Then \eqref{eqLemma1stResult01} and \eqref{eqPolynGrowthCond} imply that
\begin{align*}
&\lim_{\delta \rightarrow 0} P \Bigl(T_{1, \delta}(X) = \ell, \sup_{t \in [0, (\ell+1) \delta]} \Vert \mathbf{X}_t \Vert \leq R \Bigr)\\
&\quad \geq \, \lim_{\delta \rightarrow 0} 2 \Phi\Bigl(-\tfrac{\sqrt{\delta}}{A} \sup_{\Vert \mathbf{x} \Vert \leq R} \vert F(\mathbf{x}) \vert\Bigr)^{\ell+1} - 2 \sum_{j = 1}^\ell \Phi\Bigl(-\tfrac{\sqrt{\delta}}{A} \sup_{\Vert \mathbf{x} \Vert \leq R} \vert F(\mathbf{x}) \vert\Bigr)^j P \Bigl(\sup_{t \in I_{\ell-j, \delta}} \Vert \mathbf{X}_t \Vert > R \Bigr)\\
&\quad \geq \, \lim_{\delta \rightarrow 0} 2 \Phi\Bigl(-\tfrac{\sqrt{\delta}}{A} C (1+\delta^{-\alpha p})\Bigr)^{\ell+1} - 2 \sum_{j = 1}^\ell \Phi\Bigl(\tfrac{\sqrt{\delta}}{A}C (1+\delta^{-\alpha p})\Bigr)^j P \Bigl(\sup_{t \in [0,1]} \Vert \mathbf{X}_t \Vert > \delta^{-\alpha} \Bigr) \\
&\quad = \, \frac{1}{2^\ell}
\end{align*}
as well as 
\begin{align*}
&\lim_{\delta \rightarrow 0} P \Bigl(T_{1, \delta}(X) = \ell, \sup_{t \in [0, (\ell+1) \delta]} \Vert \mathbf{X}_t \Vert \leq R \Bigr)\\
&\quad \leq \, \lim_{\delta \rightarrow 0} 2 \Phi\Bigl(\tfrac{\sqrt{\delta}}{A} \sup_{\Vert \mathbf{x} \Vert \leq R} \vert F(\mathbf{x}) \vert\Bigr)^{\ell+1}\\
&\quad \leq \, \lim_{\delta \rightarrow 0} 2 \Phi\Bigl(\tfrac{\sqrt{\delta}}{A} C (1+\delta^{-\alpha p})\Bigr)^{\ell+1} = \frac{1}{2^\ell},
\end{align*}
and thus,
\begin{equation*}
\lim_{\delta \rightarrow 0} P \Bigl(T_{1, \delta}(X) = \ell, \sup_{t \in [0, (\ell+1) \delta]} \Vert \mathbf{X}_t \Vert \leq R \Bigr) = \frac{1}{2^\ell}.
\end{equation*}

For the claim in the case of $\ell = \infty$ let $\alpha \in (0, (2p)^{-1})$ again. For every $\delta > 0$ it holds that
\begin{equation*}
\lim_{\ell \rightarrow \infty} P(T_{1,\delta}(X) \geq \ell) \leq P(T_{1,\delta}(X) \geq \lfloor \delta^{-1} \rfloor)
\end{equation*}
and, by the continuity of $P$ from below,
\begin{equation*}
\lim_{\ell \rightarrow \infty} P(T_{1,\delta}(X) \geq \ell) = P(T_{1,\delta}(X) = \infty).
\end{equation*}
Thus, 
\begin{align*}
\lim_{\delta \rightarrow 0} P(T_{1,\delta}(X) = \infty) &\leq \lim_{\delta \rightarrow 0} P(T_{1,\delta}(X) \geq \lfloor \delta^{-1} \rfloor) \\
&= \lim_{\delta \rightarrow 0} \biggl[ P \Bigl(T_{1,\delta}(X) \geq \lfloor \delta^{-1} \rfloor, \sup_{t \in [0,1]} \Vert \mathbf{X}_t \Vert > \delta^{-\alpha} \Bigr) \\
&\hspace{3.5em} + P \Bigl(T_{1,\delta}(X) \geq \lfloor \delta^{-1} \rfloor, \sup_{t \in [0,1]} \Vert \mathbf{X}_t \Vert \leq \delta^{-\alpha} \Bigr) \biggr].
\end{align*}
Due to the continuous paths of $\mathbf{X}$ being bounded on the compact set $[0,1]$, it holds that
\begin{equation*}
P \Bigl(T_{1,\delta}(X) \geq \lfloor \delta^{-1} \rfloor, \sup_{t \in [0,1]} \Vert \mathbf{X}_t \Vert > \delta^{-\alpha} \Bigr) \leq P \Bigl( \sup_{t \in [0,1]} \Vert \mathbf{X}_t \Vert > \delta^{-\alpha} \Bigr) \xrightarrow{\delta \rightarrow 0} 0.
\end{equation*}
Furthermore, \eqref{eqLemma1stResult02} with $\ell = \lfloor \delta^{-1} \rfloor$ and $R = R_{\delta} = \delta^{-\alpha}$ and \eqref{eqPolynGrowthCond} give
\begin{align*}
P \Bigl(T_{1,\delta}(X) \geq \lfloor \delta^{-1} \rfloor, \sup_{t \in [0,1]} \Vert \mathbf{X}_t \Vert \leq \delta^{-\alpha} \Bigr) &\leq 2 \Phi \Bigl( -\tfrac{\sqrt{\delta}}{A} \sup_{\Vert \mathbf{x} \Vert \leq R} \vert F(\mathbf{x}) \vert \Bigr)^{\lfloor \delta^{-1} \rfloor} \\
&\leq 2 \Phi \Bigl( \sqrt{\delta}{A} C(1+\delta^{-\alpha p}) \Bigr)^{\lfloor \delta^{-1} \rfloor} \xrightarrow{\delta \rightarrow 0} 0.
\end{align*}
Altogether,
\begin{equation*}
0 \leq \lim_{\delta \rightarrow 0} P(T_{1,\delta}(X) = \infty) \leq 0.
\end{equation*}
\end{proof}

With the ergodicity of the given Markov chain $(X_{n \delta})_{n \in \N_0}$, we can prove that all occurrence times $R_{j,\delta}(X)$ are almost surely finite.

\begin{lemma}
\label{LemmaFiniteOccurrenceTimes}
Suppose that Assumption~\ref{Assumptions} holds and that \eqref{eqGenSDEIntro} admits an invariant measure $\mu$. Let $\mathbf{X}_0 \sim \mu$ be independent of $\mathbf{W}$, let $\mathbf{X} = (\mathbf{X}_t)_{t \geq 0}$ be the corresponding solution process to \eqref{eqGenSDEIntro} with $d$-th component $X = (X_t)_{t \geq 0}$, and assume that the ergodicity condition \eqref{eqErgoCond} is fulfilled. Then there exists a $\tilde{\delta} > 0$ such that 
\begin{equation*}
P(R_{j,\delta}(X) < \infty) = 1
\end{equation*}
for all $\delta \in (0, \tilde{\delta}]$ and all $j \in \N$.
\end{lemma}

\begin{proof}
Let $\delta >0$ be arbitrary but fixed. By assumption, the discrete-time process $(\mathbf{X}_{n \delta})_{n \in \N}$ is stationary and by \eqref{eqErgoCond} and ergodic. By \cite[Corollary 4.2, Proposition 4.3]{Krengel1985}, the process $(\mathbbm{1}_{(-\infty, 0)}(\Delta X_{n+1, \delta} \Delta X_{n, \delta}))_{n \in \N}$ then also is stationary and ergodic.
Thus, \cite[Theorem~4.4]{Krengel1985} gives
\begin{equation*}
\frac{1}{N} \sum_{n=1}^N \mathbbm{1}_{(-\infty, 0)}(\Delta X_{n+1, \delta} \Delta X_{n, \delta}) \xrightarrow{N \rightarrow \infty} P(\Delta X_{2, \delta} \Delta X_{1, \delta} < 0) \quad P\text{-a.s.}
\end{equation*}
By definition of the first occurrence time, $P(\Delta X_{2, \delta} \Delta X_{1, \delta} < 0) = P(T_{1, \delta}(X) = 1)$, and by Theorem \ref{Thm1stResult} there exits a $\tilde{\delta} > 0$ such that $P(T_{1, \delta}(X) = 1) > 0$ for all $\delta \leq \tilde{\delta}$. Thus, for $\delta \leq \tilde{\delta}$, $\mathbbm{1}_{(-\infty, 0)}(\Delta X_{n+1, \delta} \Delta X_{n, \delta})$ is infinitely often positive a.s., or equivalently, $\Delta X_{n+1, \delta} \Delta X_{n, \delta}$ is infinitely often negative a.s. Hence, for $\delta \leq \tilde{\delta}$, with probability $1$ there is an infinite number of changes in the signs of consecutive increments $\Delta X_{n+1, \delta}, \Delta X_{n, \delta}$, which, by definition, are the occurrence times $R_{j,\delta}(X)$.
\end{proof}

Now we proceed with the idea in \cite[Theorem 1]{Moy1959} to prove our second main result.

\begin{proof}[Proof of Theorem \ref{Thm2ndResult}]
By Lemma \ref{LemmaFiniteOccurrenceTimes} there exists a $\tilde{\delta} > 0$ such that $P(T_{1,\delta}(X) = 1) > 0$ and $P(R_{j,\delta}(X) < \infty) = 1$ for all $\delta \in (0, \tilde{\delta}]$ and all $j \in \N$. Now let $\delta \in (0, \tilde{\delta}]$ be arbitrary but fixed. By the same arguments as in the proof of Lemma \ref{LemmaFiniteOccurrenceTimes}, it holds that
\begin{equation*}
M_N \coloneq \frac{1}{N} \sum_{n=1}^N \mathbbm{1}_{(-\infty, 0)}(\Delta X_{n+1} \Delta X_{n}) \xrightarrow{N \rightarrow \infty} P(\Delta X_{2} \Delta X_{1} < 0) \quad P\text{-a.s.}
\end{equation*}
Since $\delta \leq \tilde{\delta}$, the sequence of occurrence times $(R_j)_{j \in \N} \subseteq \N$ is a strictly monotonically increasing sequence with $R_j \xrightarrow{j \rightarrow \infty} \infty$, and therefore $(M_{R_j})_{j\in \N}$ is a subsequence of $(M_N)_{N \in \N}$. Hence, it holds that
\begin{equation*}
M_{R_j} = \frac{1}{R_j} \sum_{n=1}^{R_j} \mathbbm{1}_{(-\infty, 0)}(\Delta X_{n+1} \Delta X_n) \xrightarrow{j \rightarrow \infty} P(\Delta X_{2} \Delta X_{1} < 0) \quad P\text{-a.s.}
\end{equation*}
The sum $\sum_{n=1}^{R_j} \mathbbm{1}_{(-\infty, 0)}(\Delta X_{n+1} \Delta X_n)$ counts how often the product of consecutive increments up to index $R_j$ is negative. By definition, $R_j$ is also the index for which the product of increments is negative for the $j$-th time. Thus, $\sum_{n=1}^{R_j} \mathbbm{1}_{(-\infty, 0)}(\Delta X_{n+1} \Delta X_n) = j$ and therefore
\begin{equation*}
\frac{j}{R_j}\xrightarrow{j \rightarrow \infty} P(\Delta X_{2} \Delta X_{1} < 0) \quad P\text{-a.s.}
\end{equation*}
Note that $P(\Delta X_2 \Delta X_1 < 0) = P(T_1 = 1) > 0$ since $\delta \leq \tilde{\delta}$, and therefore,
\begin{equation*}
\frac{1}{j} R_j \xrightarrow{j \rightarrow \infty} \frac{1}{P(\Delta X_{2} \Delta X_{1} < 0)} \quad P\text{-a.s.}
\end{equation*}
Because of $T_{1} = R_{1}$ and $T_j = R_j-R_{j-1}$, $j \in \N_{\geq 2}$, it holds that $R_j = \sum_{i=1}^j T_i$, and therefore
\begin{equation}
\label{eqProof2ndMainResult}
\frac{1}{j} \sum_{i=1}^j T_i(X) \xrightarrow{j \rightarrow \infty} \frac{1}{P(\Delta X_{2} \Delta X_{1} < 0)} \quad P\text{-a.s.}
\end{equation}
By Lemma \ref{LemmaAlmostSurelyPathspace} applied to
\begin{equation*}
B \coloneq \biggl\lbrace (\mathbf{x}_n)_{n \in \N_0} \in (\R^d)^{\N_0}: \lim_{j \rightarrow \infty} \frac{1}{j} \sum_{i=1}^j T_i((\mathbf{x}_n)_{n \in \N_0}) = P(\Delta X_{2} \Delta X_{1} < 0)^{-1} \biggr\rbrace,
\end{equation*}
the mean value $j^{-1} \sum_{i=1}^j T_i(X^{\mathbf{x}})$ converges $P$-almost surely to $P(\Delta X_{2} \Delta X_{1} < 0)^{-1}$ also for $\mu$-almost all $\mathbf{x} \in \R^d$.
Furthermore, $P(\Delta X_{2, \delta} \Delta X_{1, \delta} < 0) = P(T_{1, \delta}(X) = 1)  \xrightarrow{\delta \rightarrow 0} 1/2$ independent of the initial distribution of $\mathbf{X}$ by Theorem \ref{Thm1stResult} and thus it holds that $P(\Delta X_{2,\delta} \Delta X_{1,\delta} < 0)^{-1} \xrightarrow{\delta \rightarrow 0} 2$, which together with \eqref{eqProof2ndMainResult} proves the assertion.
\end{proof}

\section{The estimation method}
\label{SectionTheEstimationMethod}

In the SDE model \eqref{eqGenSDEIntro}, i.e.,
\begin{equation*}
\text{d}\mathbf{X}_t = \mathbf{F}(\mathbf{X}_t) \, \text{d}t + \mathbf{A} \, \text{d}\mathbf{W}_t, 
\end{equation*}
$t \in \R_0^+$, having $d$-th component \eqref{eqRelevantSDEIntro}, i.e.,
\begin{equation*}
\text{d} X_t \coloneq \text{d}\mathbf{X}^{(d)}_t = F(\mathbf{X}_t) \, \text{d}t + A \, \text{d} W_t,
\end{equation*}
we want to estimate the diffusion constant $A$ based on available time series data $x_0, x_{\Delta},\ldots,$ $x_{N \Delta} \in \R$, $N \in \N$, where $\Delta > 0$ is the step size. As before, suppose that Assumption \ref{Assumptions} holds. Regardless of the origin of the available data, we want to use the classical estimator based on the quadratic variation of the path given by
\begin{equation}
\label{eqEstimator}
\hat{A} \coloneq \sqrt{\frac{\sum_{n=0}^{N-1} (x_{(n+1)\Delta}-x_{n \Delta})^2}{\Delta N}}.
\end{equation}
Because of the data-model compatibility being restricted to specific scales in many applications as described in the introduction, our method aims to find an appropriate scale and then subsample the data accordingly. To do this, we take advantage of the fact that Theorem \ref{Thm1stResult} tells us what the distribution of monotone runs in the target model \eqref{eqRelevantSDEIntro} looks like on the infinitesimal scale. The idea is to compare characteristics of the empirical distribution of monotone runs of available data for a range of subsampling rates with the theoretical counterparts of the target model. A particularly suitable indicator is the arithmetic mean of the lengths of monotone runs, as we know its limit for the infinitesimal scale of the model under the assumption of ergodicity according to Theorem \ref{Thm2ndResult}. 

Let $k \ll N$ be a fixed subsampling factor. Since we want to use all available information, we consider all $k$ disjoint subsets of the original data set in which the data have a step size of $k \Delta$, namely
\begin{align*}
\mathcal{X}_0 &\coloneq \lbrace x_0, x_{k \Delta}, \ldots, x_{k \lfloor N/k \rfloor \Delta} \rbrace, \, \mathcal{X}_1 \coloneq \lbrace x_{\Delta}, x_{(1+k) \Delta}, \ldots, x_{(1+k \lfloor (N-1)/k \rfloor) \Delta} \rbrace, \ldots ,\\
\mathcal{X}_{k-1} &\coloneq \lbrace x_{(k-1)\Delta}, x_{(2k-1) \Delta}, \ldots, x_{(k-1+k \lfloor (N-k+1)/k \rfloor) \Delta} \rbrace.
\end{align*}
For each subset $\mathcal{X}_i$, the recurrence times $T_j(\mathcal{X}_i)$ and thus its monotone runs can be determined. Let $M(k)$ be the total number of monotone runs of all sub-datasets. There are $N+1-k$ increments in total, so the mean length of monotone runs of the entire data set for the subsampling factor $k$ is given by
\begin{equation*}
L(k) \coloneq \frac{N+1-k}{M(k)}.
\end{equation*}
This procedure can be performed for a range of subsampling factors $1,\ldots,k_{\text{max}}$, where $k_{\text{max}} \ll N$ should be such that $k_{\text{max}} \Delta$ still separates the stochastic and deterministic dynamics in \eqref{eqGenSDEIntro} from each other. For data originating from the target model \eqref{eqRelevantSDEIntro}, i.e., being a discrete-time sampling of its solution process, and $N$ sufficiently large, the mean length of monotone runs should be approximately 2, according to our analytical result given by Theorem \ref{Thm2ndResult}. So, in order to fit the target model to the available data, we search for subsampling factors $k \in \lbrace 1,\ldots, k_{\text{max}} \rbrace$ for which the mean length of monotone runs $L(k)$ is close to $2$. We adjust the estimator $\hat{A}$ given by \eqref{eqEstimator} for the diffusion constant so that it takes into account all increments with step size $k \Delta$ for a subsampling factor $k$ by defining
\begin{equation}
\label{eqEstimatorDepk}
\hat{A}(k) \coloneq \sqrt{\frac{\sum_{n=0}^{N-k} (x_{(n+k)\Delta}-x_{n \Delta})^2}{k\Delta (N+1-k)}}.
\end{equation}
Thus, if we have an optimal subsampling factor $k^* \in \lbrace 1, \ldots, k_{\text{max}} \rbrace$ with $L(k^*) \approx 2$, we can calculate the corresponding estimate $\hat{A}(k^*)$.

In general, we are not interested in a single $k$ for which $\vert L(k) -2 \vert$ is minimal, but rather in intervals for $k$ for which $L(k)$ lies within a neighborhood of $2$ with a small radius $r$. By choosing $k^*$ as the starting point of a suitable interval, we want to ensure that $k^*$ is not chosen on the basis of an outlier of $L$. We denote the length of the interval starting at $k$ and corresponding to an as long as possible tube for the curve of $L$ around $2$ with radius $r$ by $l \coloneq l(k,r)$. In summary, we are looking for a suitable subsampling factor $k$ that marks the beginning of a long tube with a small radius around $2$ in which the mean length of monotone runs remains. Figure \ref{figMethodScheme} illustrates this concept schematically. Of course, there is a trade-off between small tube radius and long tube length. To determine an optimal subsampling factor $k^*$, we therefore suggest the following procedure. First, we determine all Pareto-efficient combinations consisting of an as small as possible subsampling factor (and starting point of a tube) $k$, an as small as possible tube radius $r$, and an as large as possible corresponding tube length $l$. From the set $\mathcal{P}$ of these Pareto-efficient combinations, we want to use an automated process to select one that represents a balance between tube radius $r$ and its length $l$. To this end, we maximize the harmonic mean of the inverse tube radius and the tube length, i.e., the function
\begin{equation*}
h(r,l ) \coloneq \frac{2}{r+\tfrac{1}{l}} = \frac{2l}{rl+1}
\end{equation*}
on the set $\mathcal{P}$. The subsampling factor of this combination then is $k^*$.

\begin{figure}[t]
\begin{center}
\includegraphics[width=0.8\linewidth]{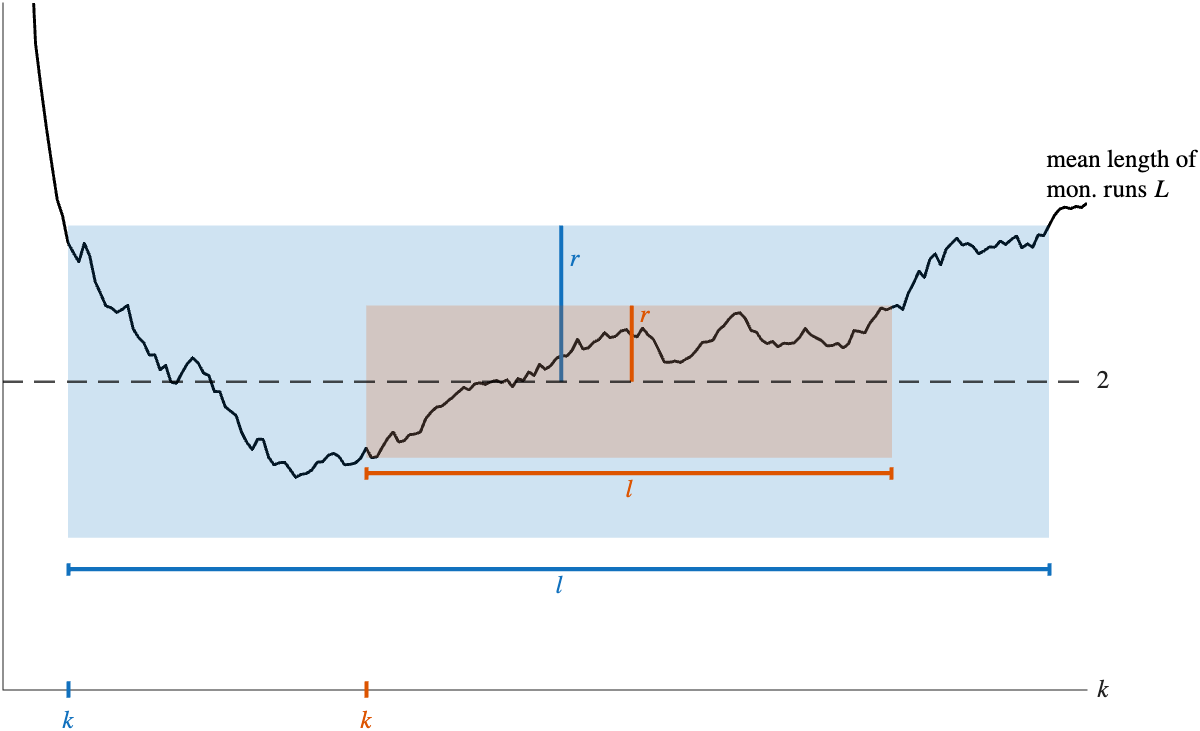} 
\end{center}
\caption{Mean length of monotone runs $L$ in dependence on the subsampling factor $k$ and two sets of exemplary Pareto-efficient combinations of subsampling factors $k$, corresponding tube radii $r$ and tube lengths $l = l(k,r)$.}
\label{figMethodScheme}
\end{figure}

Our method is summarized by the following algorithm.

\begin{algorithm}[Subsampling rate selection and estimation of the diffusion constant] 
\label{AlgorithmFullProcedure}
\ \\
\\
\textbf{Output:} Estimate $\hat{A}^*$ for diffusion constant $A$, corresponding optimal subsampling factor $k^*$, corresponding tube radius $r^*$, corresponding tube length $l^*$. \\
\ \\
\textbf{Input:} Data $\lbrace x_{n \Delta}, n = 0,\ldots,N \rbrace$ with step size $\Delta$,  maximal subsampling factor $k_{\text{max}}$.

\begin{itemize}
\item[1.] For every subsampling factor $k = 1,...,k_{\text{max}}$ consider the sub-datasets 
\begin{equation*}
\mathcal{X}^{(k)}_j \coloneq \lbrace x_{j \Delta}, x_{(j+k) \Delta}, \ldots, x_{(j+k \lfloor (N-j)/k \rfloor) \Delta} \rbrace,
\end{equation*}
$j = 0,\ldots,k-1$, determine the total number $M(k)$ of the monotone runs in all $\mathcal{X}^{(k)}_j$ and determine the corresponding mean length
\begin{equation*}
L(k) = \frac{N+1-k}{M(k)}.
\end{equation*}
\item[2.] Determine the set $\mathcal{P}$ of Pareto-efficient combinations of subsampling factors $k$, tube radii $r$ and tube lengths $l(k,r)$, such that $L(k+i) \in [2-r,2+r]$ for $i = 0,1,\ldots,l(k,r)-1$ (here, $k$ and $r$ should be small while $l(k,r)$ should be large).
\item[3.] Determine
\begin{equation*}
(k^*,r^*,l^*) = \argmax_{(k,r,l) \in \mathcal{P}} h(r,l) = \argmax_{(k,r,l) \in \mathcal{P}} \frac{2\ell}{rl +1}.
\end{equation*}
\item[4.] Calculate
\begin{equation*}
\hat{A}^* \coloneq \hat{A}(k^*) = \sqrt{\frac{\sum_{n=0}^{N-k^*} (x_{(n+k^*)\Delta}-x_{n \Delta})^2}{k^*\Delta (N+1-k^*)}}.
\end{equation*}
\item[5.] Return $\hat{A}^*, k^*, r^*, l^*$.
\end{itemize}
\end{algorithm}

\begin{remark}[Alternative indicators]
\label{RemarkAlternativeIndicators}
Since Theorem \ref{Thm1stResult} gives the full asymptotic distribution of the first length of monotone runs, one can construct additional indicators for the optimal subsampling factor $k^*$. An apparent alternative to the mean length of monotone runs is to consider the probability for monotone runs having a certain length. Let $m \in \N$, $\delta > 0$ and $X$ be a solution process of the target model \eqref{eqRelevantSDEIntro} in the sense of Assumption \ref{Assumptions}, then, by Theorem \ref{Thm1stResult}, for the probability that the first monotone run has a length of $m$ it holds that
\begin{equation*}
P(T_{1,\delta}(X) = m) \xrightarrow{\delta \rightarrow 0} \frac{1}{2^m}.
\end{equation*}
On the other hand, for a data set $x_0, x_{\Delta}, \ldots, x_{N \Delta}$ or for each subset at step size $k$, i.e., for each $\mathcal{X}^{(k)}_j = \lbrace x_{j \Delta}, x_{(j+k) \Delta}, \ldots, x_{(j+k \lfloor (N-j)/k \rfloor) \Delta} \rbrace$, the relative frequency of monotone runs of length $m$ can be determined. It is reasonable to compare the relative frequency analogously to the mean length of the monotone runs as a function of the subsampling factor $k$ with the theoretical value of $2^{-m}$ of the target model and to select a $k$ for which both quantities show good agreement. The general idea of comparing empirical parameters with parameters of the distribution of the first length of monotone runs in the target model may also be extended to variance, for example. Such considerations could possibly be extended to the construction of asymptotic confidence intervals for the expected value, but will not be discussed further in this paper.
\end{remark}

\begin{remark}[Extrema quadratic variation]
\label{RemarkExtQV}
Closely linked to the consideration of monotone runs is the extrema quadratic variation (ExtQV) of the sample, which was proposed and investigated in \cite{Manikas2018,Manikas2019}. Instead of squaring each individual increment in the data, as in the estimator based on quadratic variation, one can determine and square the total increment of each monotone run. The sum of these squared monotone run increments divided by the length $T$ of the time period over which the sampled process was observed yields the (relative) ExtQV of the sample. However, in \cite{Manikas2018,Manikas2019} the application to the homogenization case is considered and the estimation of the diffusion constant in the limit $\varepsilon \rightarrow 0$ and $\delta \rightarrow 0$ is investigated. Consequently, this approach is particularly aimed at situations involving very fine-grained data. Instead, we propose the following modification, regardless of the origin and resolution of the data. For a realization of a Brownian motion path the limit of the expected ExtQV, as $\delta \rightarrow 0$, is given by $1+\tfrac{4}{\pi}$, see \cite[Proposition 7.1]{Manikas2018}. Provided that this limit value remains valid in the case of a nonzero drift, the idea is to correct the ExtQV of the available data by the factor $(1+\tfrac{4}{\pi})^{-1}$ in order to obtain the squared diffusion constant $A^2$ and to use the resulting estimator in a similar way as $\hat{A}$ in the procedure described in Algorithm \ref{AlgorithmFullProcedure}. The key difference to the approach in \cite{Manikas2018,Manikas2019} is therefore that we want to use the factor $(V^{\infty})^{-1}$ to apply the ExtQV for estimating $A^2$ to situations that exceed $\varepsilon \rightarrow 0$ in the homogenization setting. In the present article, we do not investigate this approach further.
\end{remark}

\section{Application to fiber lay-down processes}
\label{SectionApplication}

In this section, we present an application of our method to a problem in the industrial production of nonwoven textiles. In the production process a fiber is spun from a molten granular and then entangled by acting turbulent air flows while laying down on a moving conveyer belt. Stochastic models based on SDEs for the lay-down curves of the fibers on the conveyor belt allow efficient numerical fiber simulations and in this context replace numerically complex models based on partial differential equations. Coupling these different models requires adapting the parameters of the stochastic model. This application is simulated fiber data is available and a suitable homogenization setting is known, with the help of which the method can be systematically investigated.

\subsection{Stochastic fiber lay-down models}
\label{SubsectionIntroApplication}

In \cite{Goetz2007} the \textit{basic model}
\begin{align}
\text{d} \boldsymbol \xi_t &= \boldsymbol \tau (\alpha_t) \: \text{d}t + v \mathbf{e}_1 \: \text{d}t \nonumber \\
\text{d} \alpha_t &= - \nabla V(\boldsymbol \xi_t) \cdot \boldsymbol \tau^\bot (\alpha_t) \: \text{d}t + A \: \text{d} W_t,
\label{eqBasicModel}
\end{align}
$t \in R_0^+$, with initial conditions $\boldsymbol \xi_0 \in \R^2$ and $\alpha_0 \in [0, 2\pi]$, was formulated. Here, $\boldsymbol\xi = \boldsymbol\eta - \boldsymbol\gamma$ denotes the difference of the arclength parameterized fiber curve $\boldsymbol\eta \colon \R_0^+ \times \Omega \rightarrow \R^2$ and the reference curve $\boldsymbol\gamma$ specifying the conveyer belt movement by $\boldsymbol\gamma_t = -v t \mathbf{e}_1$, where $v = v_{\text{belt}} / v_{\text{in}} \in [0,1]$ is the ratio between the belt speed $v_{\text{belt}}$ and the speed of the fiber production $v_{\text{in}}$.
The corresponding angles of the actual fiber curve $\boldsymbol\eta$ with respect to the direction of the conveyer belt movement are given by $\alpha: \R_0^+ \times \Omega \rightarrow \R$. By $\boldsymbol\tau(\alpha) \coloneqq (\cos \alpha, \sin \alpha)^\top$, we introduce the normalized tangent on the fiber, so that $\boldsymbol\tau^{\bot}(\alpha) \coloneqq (-\sin \alpha, \cos \alpha)^\top$ is the corresponding orthonormal polar unit vector. The drift $-\nabla V(\boldsymbol \xi) \cdot \boldsymbol \tau^\bot (\alpha)$ ensures that the fiber tends back to the origin, where the potential $V \colon \R^2 \rightarrow \R$ determines how wide the fiber can spread. A typical choice is 
\begin{equation*}
V(\boldsymbol\xi) = \frac{1}{2} \boldsymbol\xi^\top \mathbf{C}^{-1} \boldsymbol\xi
\end{equation*}
with positive-definite matrix $\mathbf{C} \in \R^{2 \times 2}$ giving a standard model for the buckling behavior of the fiber. The random effects of the production process, e.g., resulting from the turbulent flow during the fiber spinning and lay-down, are summarized in a one-dimensional standard Brownian motion $W$ with diffusion constant $A \in \R^+$.

By choosing $\mathbf{X}_t = (\boldsymbol \xi_t, \alpha_t) = ( \xi_t^{(1)},\xi_t^{(2)}, \alpha_t)$,
\begin{equation*}
\mathbf{F}(\mathbf{x}) = \mathbf{F}(x_1,x_2,x_3) = \begin{pmatrix}
\cos(x_3) + v \\
\sin(x_3) \\
-\nabla V(x_1, x_2) \cdot (-\sin(x_3), \cos(x_3))^\top
\end{pmatrix}
\end{equation*}
and $\mathbf{A} = \text{diag}(0,0,A)$, the basic model can be expressed in the notation of our general setting. In the case of $V(\boldsymbol\xi) = \tfrac{1}{2} \boldsymbol\xi^\top \, \mathbf{C}^{-1} \boldsymbol\xi$, the drift coefficient $\mathbf{F}$ is locally Lipschitz continuous and satisfies a linear growth condition implying the existence of a unique solution in the sense of Assumption \ref{Assumptions}. Regarding the ergodicity assumption of Theorem \ref{Thm2ndResult}, we use the fact that in \cite{Kolb2013} the fiber lay-down process described by \eqref{eqBasicModel} was shown to possess the strong Markov property. In particular, every sample $((\boldsymbol\xi, \alpha)_{n \Delta})_{n \in \N_0}$, $\Delta > 0$, is a Markov chain. At least for some special cases such as $v=0$, i.e., for a stationary conveyor belt, or $\mathbf{C}= \text{diag}(\sigma^2, \sigma^2)$ in the potential $V$, the existence of a unique invariant measure is known, see, e.g., \cite[Theorem 3.1]{Kolb2013}. Since, according to \cite[Proposition 2.3]{Kolb2013}, $((\boldsymbol\xi, \alpha)_{n \Delta})_{n \in \N_0}$ is also irreducible for every $\Delta > 0$, we can conclude by Lemma \ref{LemmaIrreducibility} that these skeleton chains are ergodic in the special cases mentioned above. Numerical tests suggest that we can apply our method also in the case $v>0$ and $\mathbf{C} = \text{diag}(\sigma_1^2, \sigma_2^2)$ with $\sigma_1 \neq \sigma_2$.

In a typical application, the ratio $v$ between conveyor belt speed and production speed is known, while the parameters of the potential $V$ and the diffusion constant $A$ must be estimated. The method used to this point for estimating these parameters in the special case of $V(\boldsymbol\xi) = \tfrac{1}{2} \boldsymbol\xi^\top \, \mathbf{C}^{-1} \boldsymbol\xi$, $\mathbf{C} = \text{diag}(\sigma_1^2, \sigma_2^2)$, is described in \cite{Klar2009,Grothaus2014}. We build on this and improve the estimation of $A$ by applying the method described by Algorithm \ref{AlgorithmFullProcedure}. An alternative approach for the estimation of the parameters in the basic model is discussed in \cite{Bock2018}. It is based on expected occupation times, approximated using Monte Carlo simulations. As a consequence, the accuracy of the estimates depend on the number of sample paths while our approach is designed to be applied to data from a single path.

In addition to the advantage of being able to test our approach directly on simulated fiber data, for stochastic fiber lay-down models also a homogenization setting can be formulated. This allows data to be generated using a classical multiscale diffusion with the advantage of knowing the true values of the parameters to be estimated in the basic model, at least asymptotically.
For this purpose, the \textit{improved smooth model} can be defined by
\begin{align}
\text{d} \boldsymbol\xi_t^{(\varepsilon)} &= \boldsymbol\tau(\alpha_t^{(\varepsilon)}) \: \text{d}t +v \mathbf{e}_1 \: \text{d}t \nonumber \\ 
\text{d} \alpha_t^{(\varepsilon)} &= -\nabla V(\boldsymbol\xi_t^{(\varepsilon)}) \cdot \boldsymbol\tau^\perp(\alpha_t^{(\varepsilon)}) \: \text{d}t + \frac{\kappa_t^{(\varepsilon)}}{\varepsilon} \: \text{d}t \nonumber \\
\text{d} \kappa_t^{(\varepsilon)} &= -\frac{1}{\varepsilon^2 R} \kappa_t^{(\varepsilon)} \: \text{d}t + \frac{K}{\varepsilon} \: \text{d} W_t, \label{eqSmoothModel}
\end{align}
equipped with appropriate initial conditions $\boldsymbol\xi_0 \in \R^2$, $\alpha_0 \in [0, 2\pi]$ and $\kappa_0 \in \R$. This formulation can be found in \cite{Maringer2013} as a slightly modified version of the original formulation in \cite{Herty2009}. By the third equation the curvature $\kappa$ of the fiber curve is described. Here, $R \in \R^+$ expresses the stiffness of the fiber, where larger values imply larger loops performed by the fiber. As usual in the homogenization setting, the solution $(\boldsymbol\xi^{(\varepsilon)}, \alpha^{(\varepsilon)})$ of the improved smooth model converges weakly to the solution of the basic model as $\varepsilon \rightarrow 0$, where the diffusion constant of the basic model is given by the product of the new diffusion constant $K \in \R_0^+$ and the stiffness parameter $R$, i.e., $A = RK$. For a proof of the respective statement for the original formulation of the smooth model, see \cite{Herty2009}. 

\begin{remark}[Reconstruction of angular data]
\label{RemarkAngularRec}
In practice, typically only fiber data $\boldsymbol\eta_{n \Delta}$ is provided, where $n \in \lbrace 0,\ldots,N \rbrace$, $N \in \N$ and $\Delta > 0$. This is equivalent to having data $\boldsymbol\xi_{n \Delta}$ corresponding to the $\boldsymbol\xi$ component in the basic model, since $\boldsymbol\eta$ and $\boldsymbol\xi$ can be calculated directly from each other using the known value of the ratio $v$ between conveyor belt speed and production speed. Since the diffusion constant $A$ is part of the equation for $\alpha$, angular data $\alpha_{n \Delta}$, $n \in \lbrace 0,\ldots,N-1 \rbrace$, is required in order to estimate $A$ via $\hat{A}$ given by \eqref{eqEstimator}. For the reconstruction of angular data we proceed as follows:
\begin{itemize}
\item[1.] We determine the slopes of the secants through successive points $\boldsymbol\eta_{(n+1)\Delta}$ and $\boldsymbol\eta_{n \Delta}$, and thus obtain preliminary angles $\tilde{\alpha}_{n \Delta} \in [0, 2 \pi)$.
\item[2.] The corresponding angular increments lie in $(-2 \pi, 2 \pi)$ and lead to an overestimation of the diffusion constant, since consecutive angles exceeding the $360^\circ$ threshold resulting in increments of large magnitude are included. For this reason, we correct the preliminary angles for each increment where we suspect that the $360^\circ$ threshold has been exceeded. The best assumption we can make is that this always holds for increments with absolute value greater than $\pi$. We then shift the affected angle and all subsequent angles by $2 \pi$ so that the corresponding angular increments fall below $\pi$ in absolute value.
\end{itemize}
\end{remark}

\subsection{Numerical tests}
\label{SubsectionNumericalTests}

In this subsection, we present some observations made in numerical tests on our method (Algorithm \ref{AlgorithmFullProcedure}) based on synthetic data from Euler approximations of the basic model itself \eqref{eqBasicModel} and the improved smooth model \eqref{eqSmoothModel}. In the following we always choose $\mathbf{C} = \text{diag}(\sigma_1^2, \sigma_2^2)$, $\sigma_1, \sigma_2 \in \R^+$, for the potential $V(\boldsymbol\xi) = \tfrac{1}{2} \boldsymbol\xi^\top \, \mathbf{C}^{-1} \boldsymbol\xi$ in both models. In this context, $\sigma_1$ and $\sigma_2$ are also referred to as \textit{throwing ranges}. In addition, we work with a dimensionless formulation of the models, which are obtained by the scaled parameters $\overline{\sigma}_1 = 1$, $\overline{\sigma}_2 = \tfrac{\sigma_2}{\sigma_1}$ and $\overline{A} = \sqrt{\sigma_1}A$ or $\overline{R} = \tfrac{R}{\sigma_1}$ and $\overline{K} = \sqrt{\sigma_1^3} K$, see, for example, \cite{Klar2009}. For simplicity, we omit the bar in the following.

We generate data from \eqref{eqBasicModel} and \eqref{eqSmoothModel} using a semi-implicit Euler-Maruyama method. Specifically, for a fixed step size $\Delta > 0$, the approximations $\hat{\alpha}_{n+1}$, $n = 0,1,\ldots$, for $\alpha_{(n+1)\Delta}$ in the basic model and $\hat{\kappa}_{n+1}$ for $\kappa_{(n+1)\Delta}$ in the smooth model are calculated explicitly from $\hat{\alpha}_n$ and $\hat{\kappa}_n$, respectively, while the remaining variables $\hat{\boldsymbol\xi}_{n+1}$ and $\hat{\alpha}_{n+1}$ already rely on $\hat{\alpha}_{n+1}$ and $\hat{\kappa}_{n+1}$ respectively, and thus are treated implicitly at this point in order to achieve slight stability benefits.

Self-consistency tests, i.e., using data from an approximation of the basic model itself confirm that the method works in principle. Of particular interest is the use of the homogenization setting by combining \eqref{eqBasicModel} and \eqref{eqSmoothModel}. Here, we use semi-implicit Euler-Maruyama approximations of the smooth model \eqref{eqSmoothModel} as simulated data to estimate the corresponding limit value (for $\varepsilon \rightarrow 0$) of the diffusion constant $A = RK$ in the basic model \eqref{eqBasicModel}.

\begin{figure}[t]
\begin{center}
\includegraphics[width=0.8\linewidth]{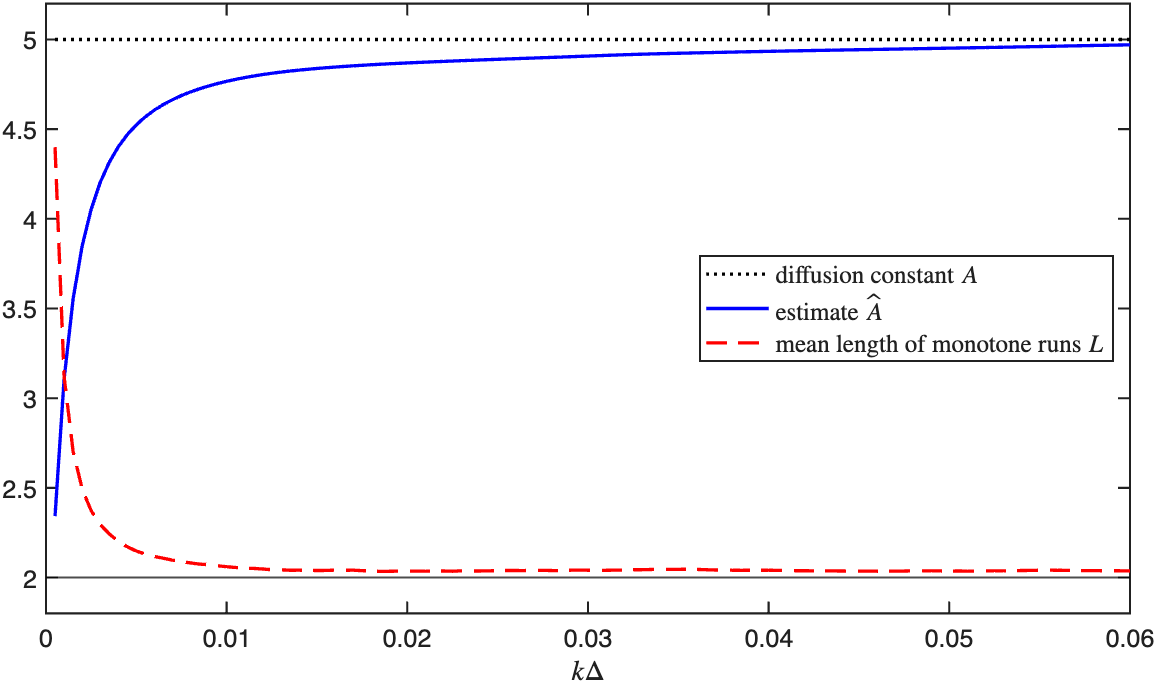} 
\end{center}
\caption{Estimates for the diffusion constant $A$ of the basic model \eqref{eqBasicModel} and mean lengths of monotone runs in dependence on the effective step size $k \Delta$ for data coming from an Euler-Maruyama approximation of the improved smooth model \eqref{eqSmoothModel} having parameters $R = K = \sqrt{5}$, $\varepsilon = 0.02$ and $v = 0.1$.}
\label{figAAMon01}
\end{figure}
\begin{figure}[t]
\begin{minipage}[h]{0.49\linewidth}
\begin{center}
\includegraphics[width=1\linewidth]{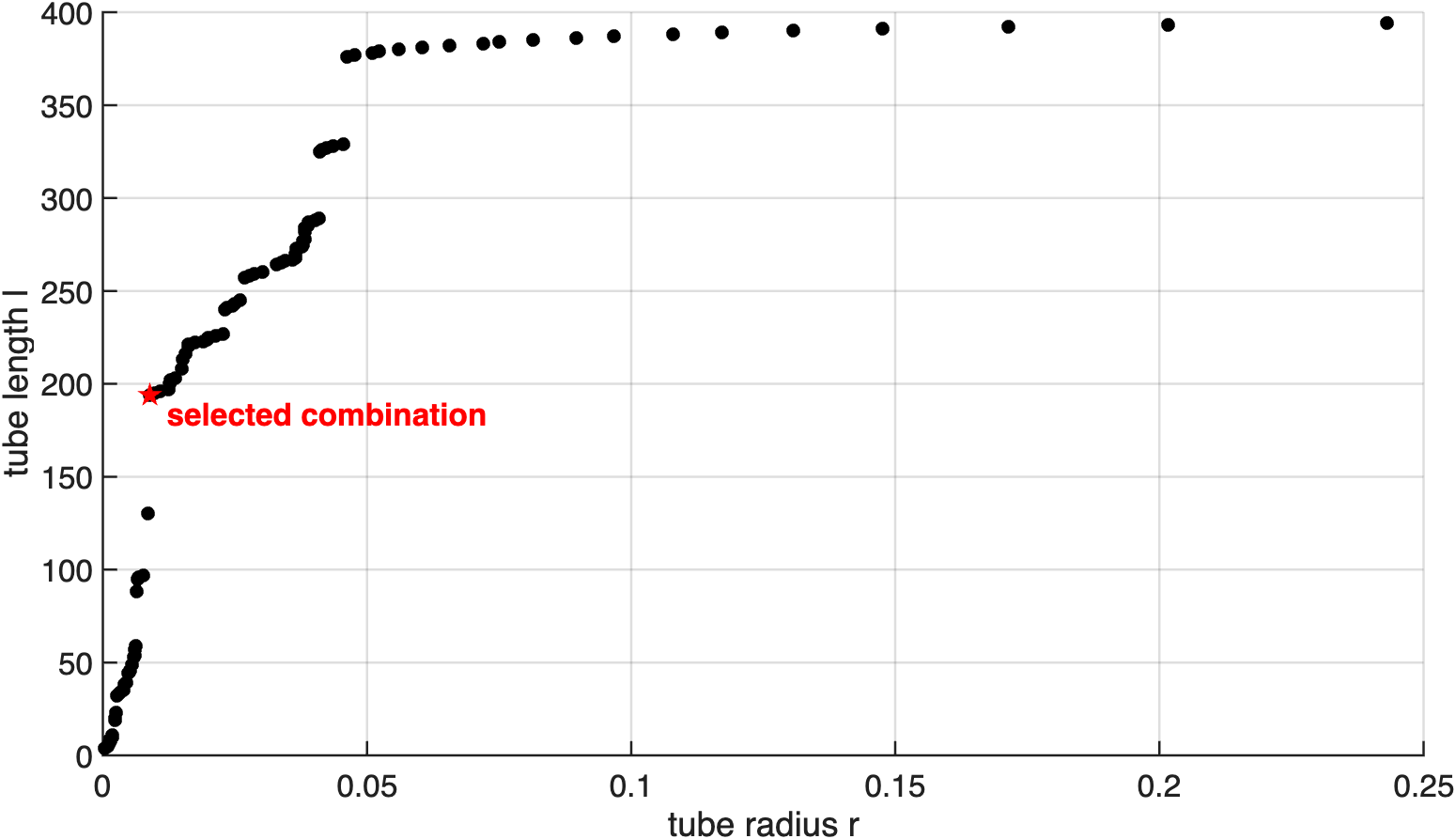} 
\end{center}
\end{minipage}
\hfill
\begin{minipage}[h]{0.49\linewidth}
\begin{center}
\includegraphics[width=1\linewidth]{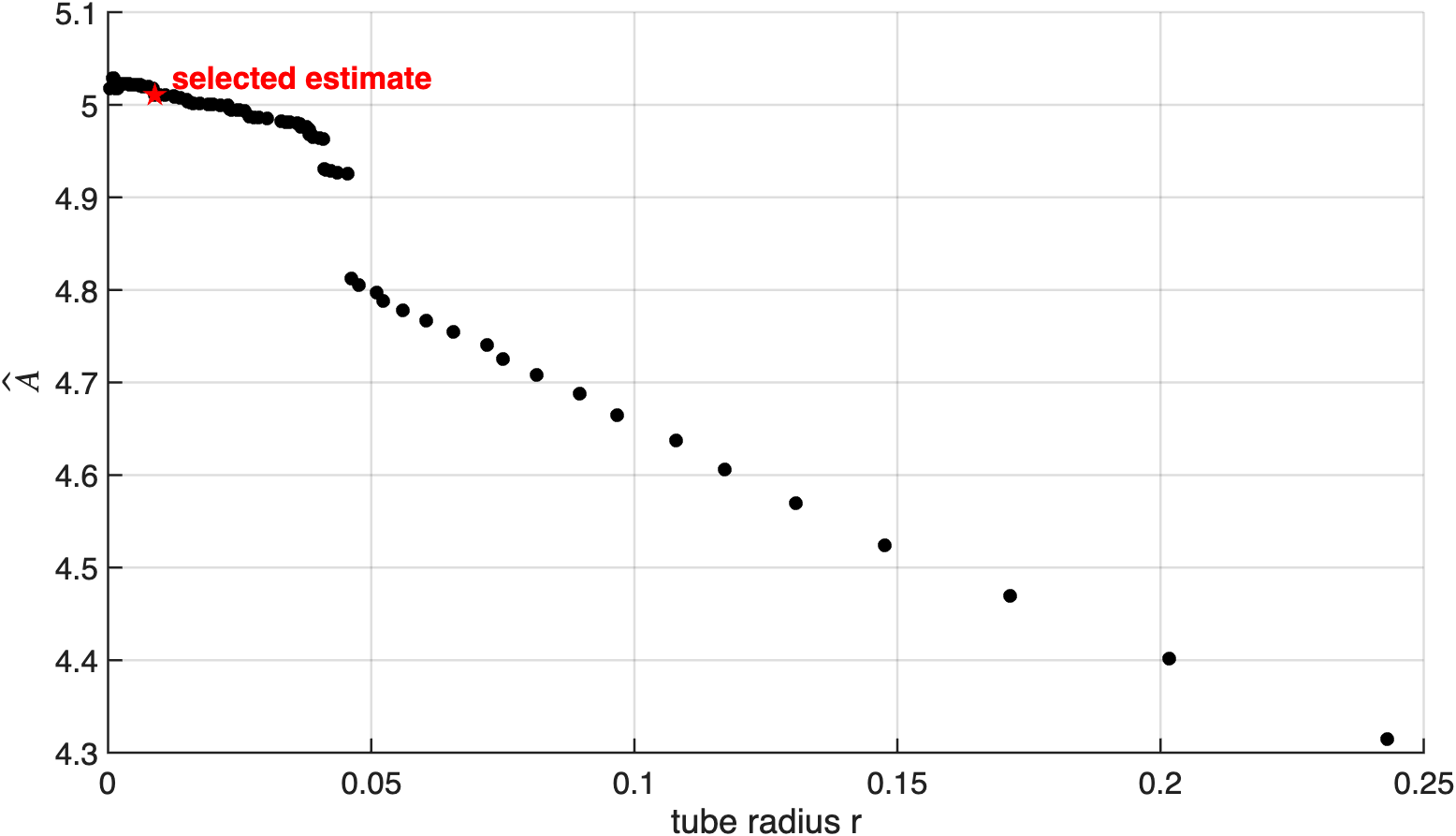} 
\end{center}
\end{minipage}
\caption{Pareto efficient combinations of tube lengths and radii around 2 for the curve of the mean length of monotone runs in the situation of Figure \ref{figAAMon01} (left) and corresponding estimates for the diffusion constant $A$ in dependence on the tube radius (right).}
\label{figParetoFront01}
\end{figure}

For the first example, we choose the parameters $\sigma_1 = \sigma_2= 1$ and $A = 5$, where the latter is composed of $R = \sqrt{A}$ and $K = \sqrt{A}$ in the improved smooth model. We set the scale separation parameter $\varepsilon = 0.02$. For the speed ratio $v$ of conveyor belt speed to production speed, we choose the value $0.1$ and assume that, unlike the other parameters, this is known. We generate $2 \times 10^8$ data points having step size $2\times 10^{-6}$ in order to have a sufficient resolution for the Euler-Maruyama approximation of the small scaled part in the multiscale diffusion. In order to make the situation realistic in terms of practical scenarios, we only use every 250th point from the Euler approximation described above, so that the initial step size of the simulated data is $\Delta = 5\times 10^{-4}$. For the same reason, we discard the angular data from the Euler-Maruyama approximation and instead reconstruct angular increments as described in Remark \ref{RemarkAngularRec}. Figure \ref{figAAMon01} shows the resulting estimates
\begin{equation*}
\hat{A}(k) = \sqrt{\frac{\sum_{n=1}^{N-k} (\alpha_{n+k}-\alpha_n)^2}{k\Delta (N-k)}}
\end{equation*}
for subsampling factors $k \leq 120$ and the corresponding mean lengths $L(k)$ of monotone runs. Here we have the typical observation that the diffusion constant is significantly underestimated without subsampling. Only with increasingly higher subsampling factors do the estimates approach the true value of the diffusion constant of the basic model. Simultaneously, the curve of the mean length of monotone runs falls to the value 2 expected on the basis of Theorem \ref{Thm2ndResult}, thus confirming that a high subsampling factor should be chosen. As described in Algorithm \ref{AlgorithmFullProcedure}, we determine the Pareto-efficient combinations of tube radii and lengths around 2 for the curve of the mean length of monotone runs and select a moderate combination by maximizing the harmonic mean of tube lengths and the corresponding inverse radius. While Figure \ref{figAAMon01} only shows the section for $k \leq 120$, we performed our method with a maximum subsampling factor of $k_{\text{max}} = 400$. The pareto efficient combinations and the selected combination, which in this example is given by $k^* = 207$, $r^* \approx 0.0089$ and $l^* = 194$ leading to an estimate of $\hat{A}^* \approx 5.0103$, are shown in the left hand side of Figure \ref{figParetoFront01}.
In fact, the exact selection from the Pareto-efficient combinations plays a minor role in the estimated value for $A$, as long as a moderate compromise is made between tube length and radius. This can be seen on the right hand side of Figure \ref{figParetoFront01}, where the estimated values for $A$ are plotted for all Pareto combinations in dependence on the tube radius.

In the first example, it appears that the largest possible subsampling factor should be selected. In general, this is not the case, as is evident in the following scenario. Here, a higher conveyor belt speed of $v = 0.4$ results in the best subsampling factor being within a small range of moderate values, see Figure \ref{figAAMon02}. Our method gives the combination $k^* = 21$, $r^* \approx 0.0276$ and $l^* = 30$, and an estimate of $\hat{A}^* \approx 4.7154$. Note that the curve of the monotone runs remains well below $RK = 5$ even at its maximum. Since $A$ and $RK$ coincide only in the limit of $\varepsilon \rightarrow 0$, this indicates that in the situation of a larger value for $v$, the scale separation parameter $\varepsilon$ should be smaller so that \eqref{eqSmoothModel} continues to be a good approximation for \eqref{eqBasicModel}.

\begin{figure}[t]
\begin{center}
\includegraphics[width=0.8\linewidth]{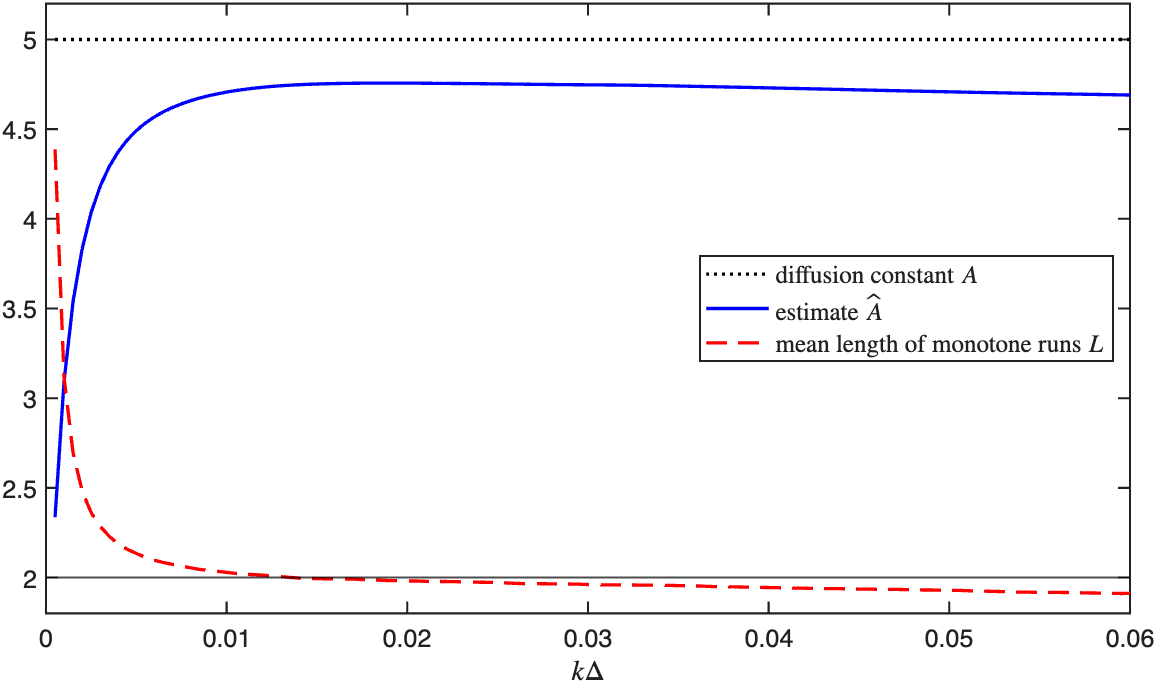} 
\end{center}
\caption{Estimates for the diffusion constant $A$ of the basic model \eqref{eqBasicModel} and mean lengths of monotone runs in dependence on the effective step size $k \Delta$ for data coming from an Euler-Maruyama approximation of the improved smooth model \eqref{eqSmoothModel} having parameters $R = K = \sqrt{5}$, $\varepsilon = 0.02$ and $v = 0.4$.}
\label{figAAMon02}
\end{figure}

\begin{figure}[t]
\begin{minipage}[h]{0.49\linewidth}
\begin{center}
\includegraphics[width=1\linewidth]{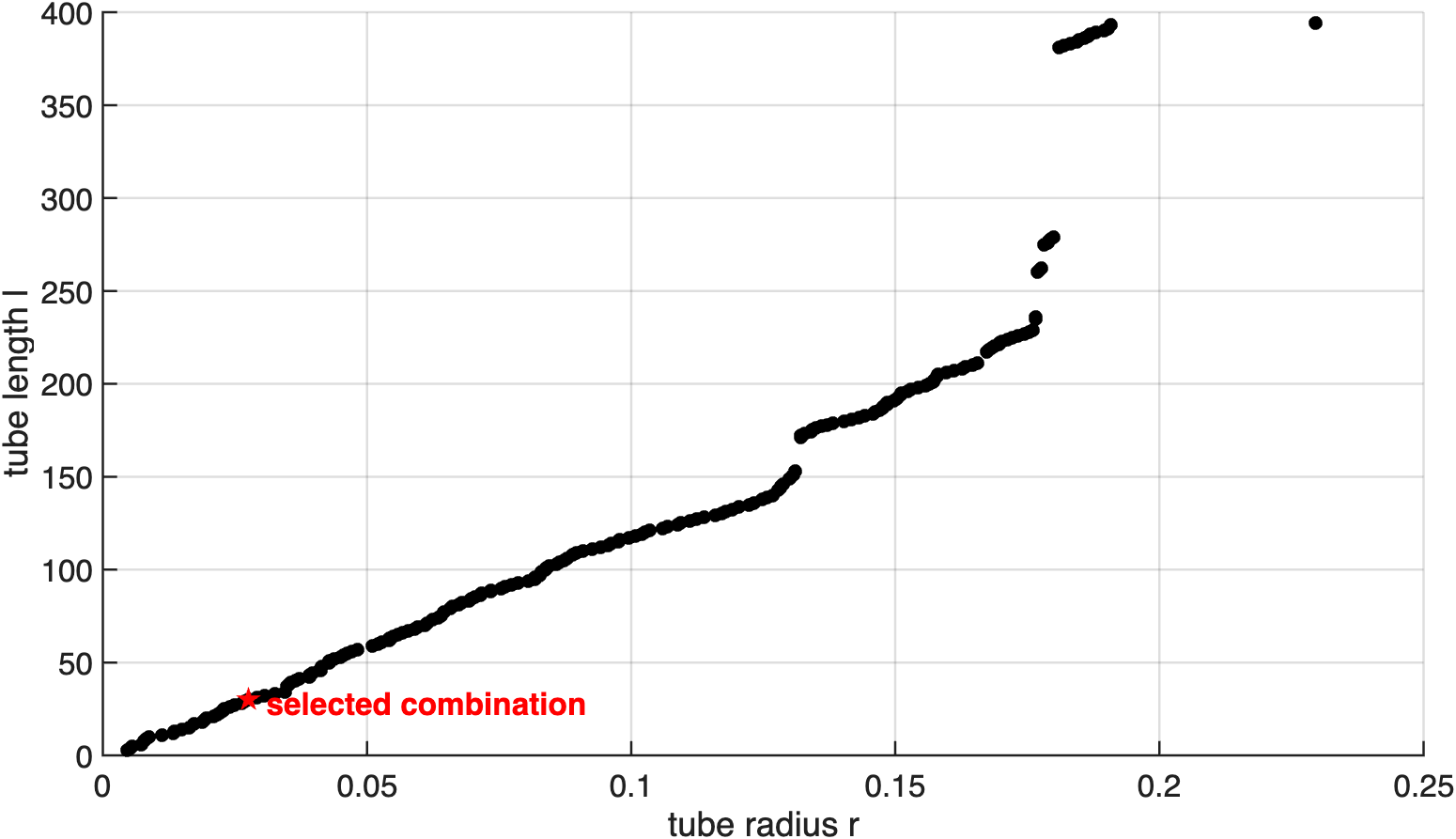} 
\end{center}
\end{minipage}
\hfill
\begin{minipage}[h]{0.49\linewidth}
\begin{center}
\includegraphics[width=1\linewidth]{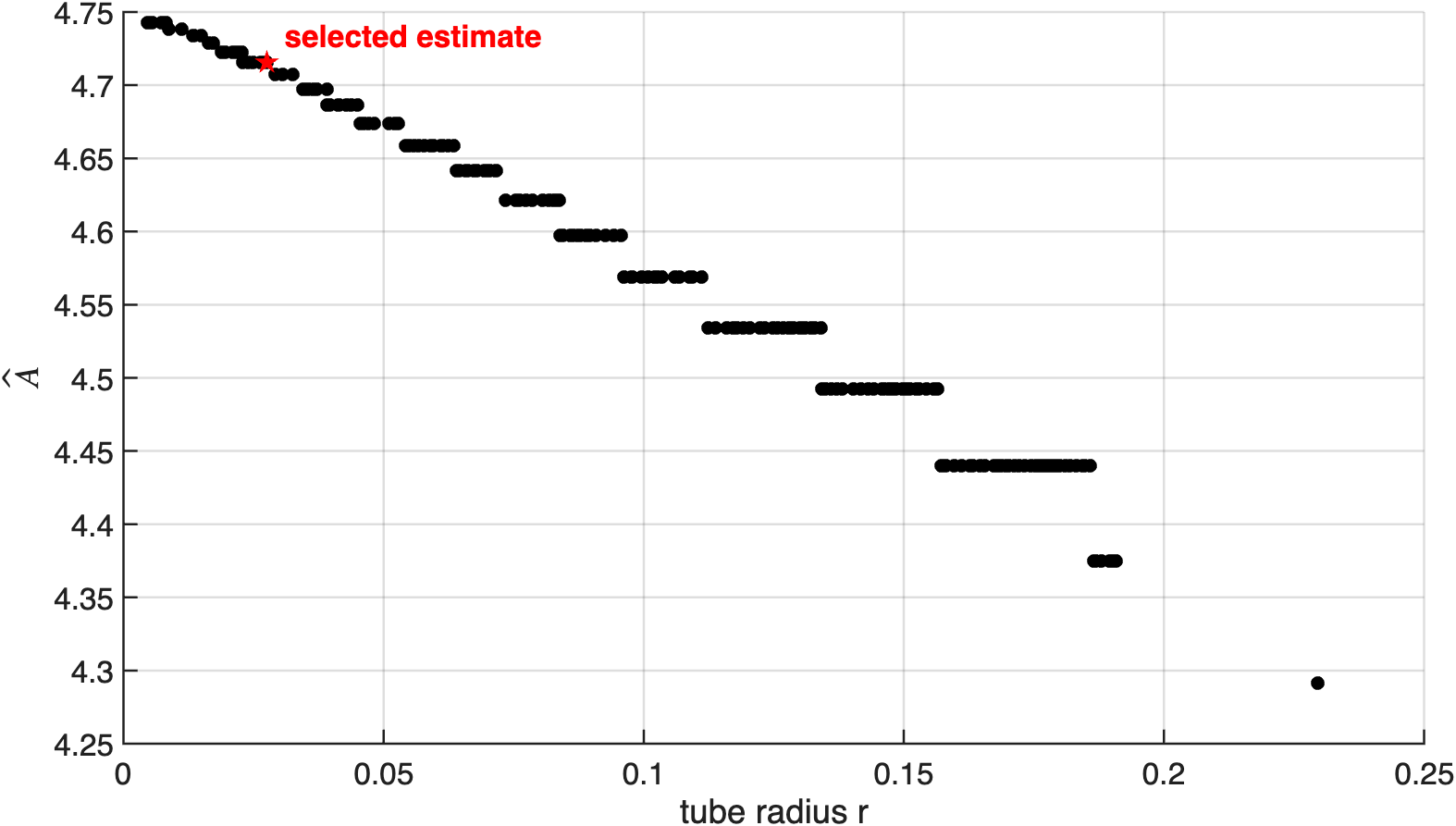} 
\end{center}
\end{minipage}
\caption{Pareto efficient combinations of tube lengths and radii around 2 for the curve of the mean length of monotone runs in the situation of Figure \ref{figAAMon02} (left) and corresponding estimates for the diffusion constant $A$ in dependence on the tube radius (right).}
\label{figParetoFront02}
\end{figure}

Our last example illustrates that in general a simple maximization of $\hat{A}$ is not effective either. In particular, it shows a situation in which too high subsampling factors lead to an overestimation of the diffusion constant. To this end, we set $A$ to $0.5$, while all other parameters remain the same as in the first example. Again, our method finds a suitable subsampling factor $k^* = 8$ with tube radius $r^* \approx 0.0854$ and corresponding tube length $l^* = 22$ giving an estimate of $\hat{A}^* \approx 0.4830$. The corresponding plots can be found in Figures \ref{figAAMon03} and \ref{figParetoFront03}.

\begin{figure}[t]
\begin{center}
\includegraphics[width=0.8\linewidth]{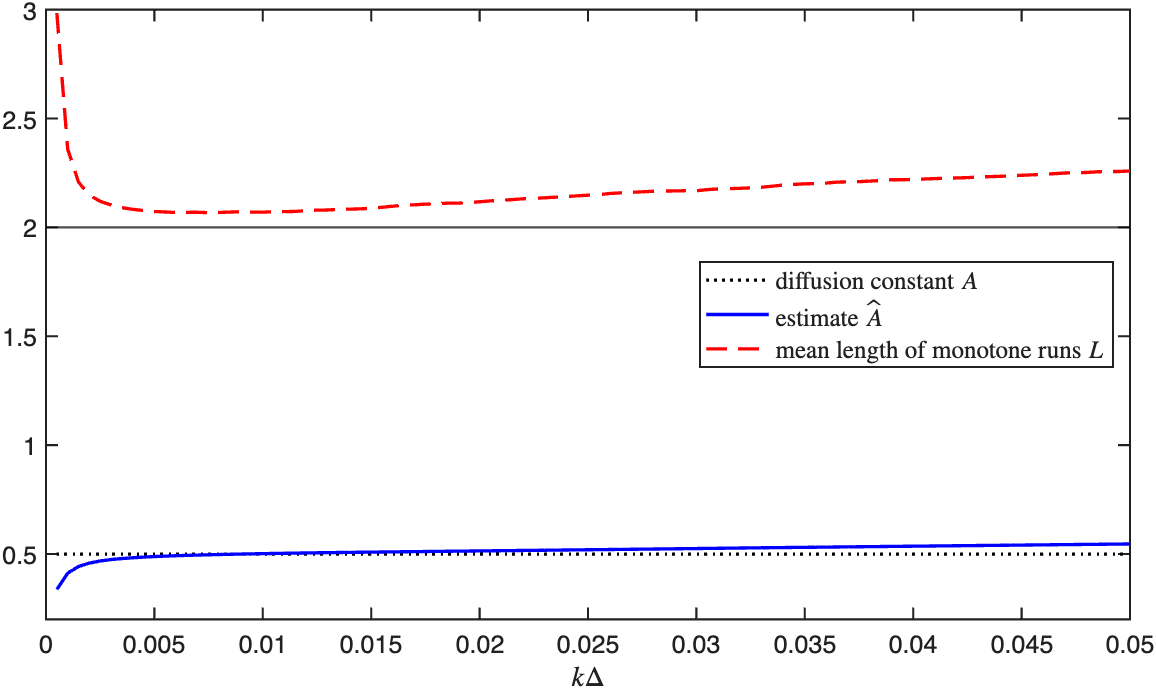} 
\end{center}
\caption{Estimates for the diffusion constant $A$ of the basic model \eqref{eqBasicModel} and mean lengths of monotone runs in dependence on the effective step size $k \Delta$ for data coming from an Euler-Maruyama approximation of the improved smooth model \eqref{eqSmoothModel} having parameters $R = K = \sqrt{0.5}$, $\varepsilon = 0.02$ and $v = 0.1$.}
\label{figAAMon03}
\end{figure}

\begin{figure}[t]
\begin{minipage}[h]{0.49\linewidth}
\begin{center}
\includegraphics[width=1\linewidth]{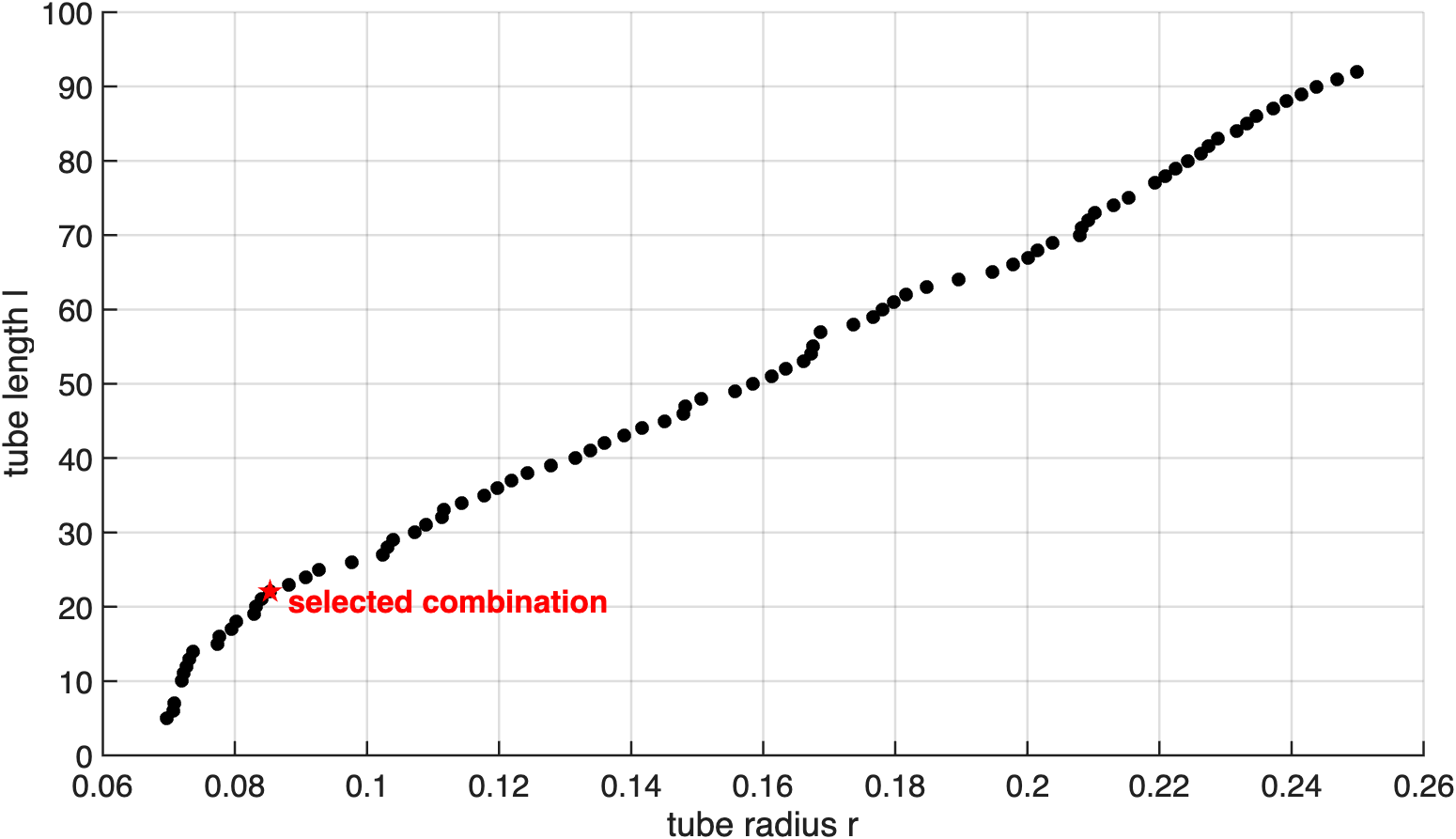} 
\end{center}
\end{minipage}
\hfill
\begin{minipage}[h]{0.49\linewidth}
\begin{center}
\includegraphics[width=1\linewidth]{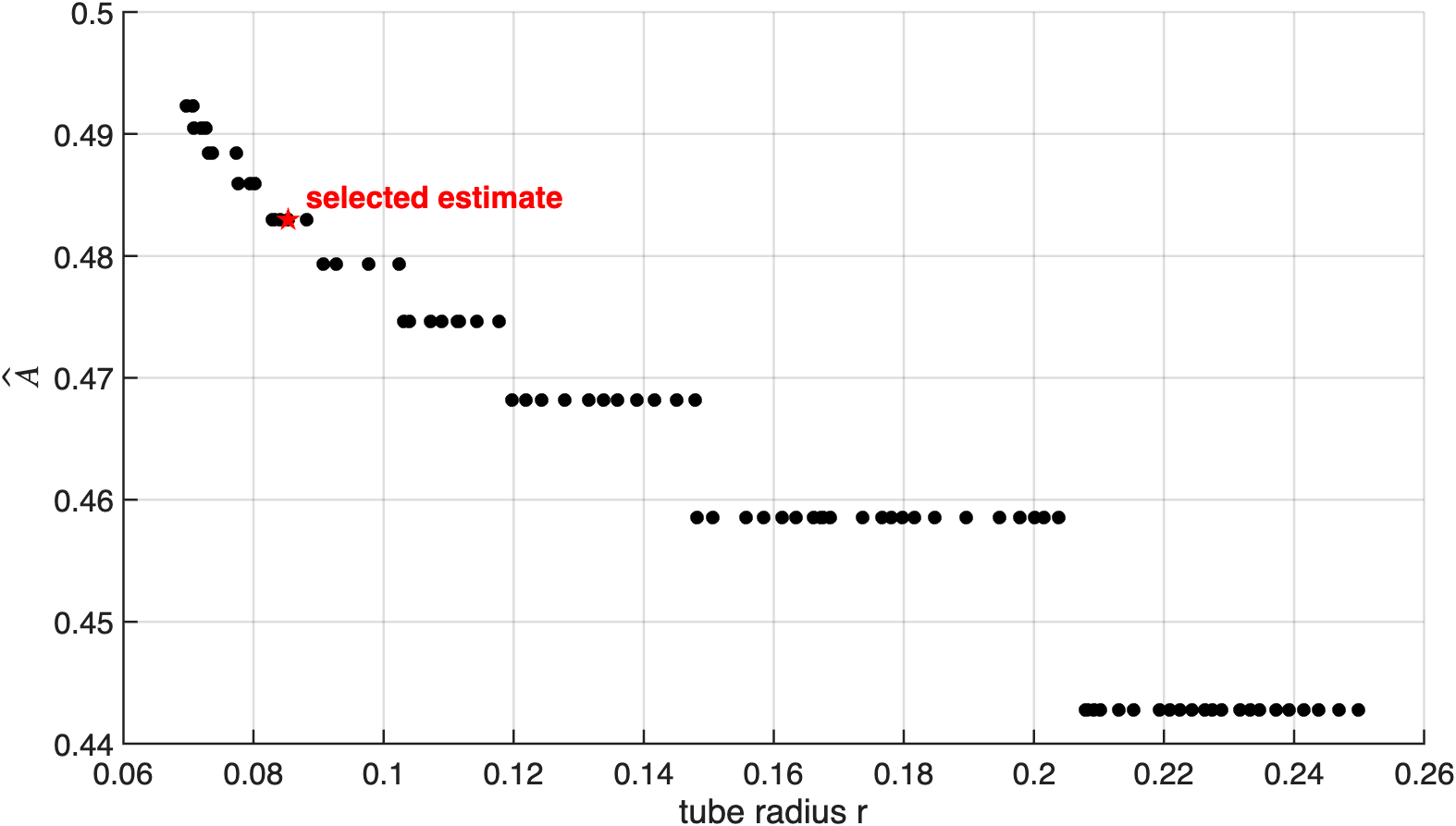} 
\end{center}
\end{minipage}
\caption{Pareto efficient combinations of tube lengths and radii around 2 for the curve of the mean length of monotone runs in the situation of Figure \ref{figAAMon03} (left) and corresponding estimates for the diffusion constant $A$ in dependence on the tube radius (right).}
\label{figParetoFront03}
\end{figure}

\subsection{Real-world application}
\label{SubsectionCompleteProcedure}

In this subsection, we address the real-world application of our method to the fiber lay-down setting. The data used here are now obtained from the FIDYST tool developed at Fraunhofer ITWM, which is based on partial differential equations and numerically simulates a single fiber for specified production parameters at comparatively high computational cost, see \cite{Klar2009,Marheineke2011,WebITWM}. In practice, the stochastic basic model \eqref{eqBasicModel} is therefore intended to be used for the efficient simulation of a large number of fibers that coincide qualitatively and quantitatively well with the computationally expensive simulated fiber data. To this end, the drift parameters as well as the diffusion constant of the basic model must be adapted to the respective simulated fiber data using an automated process.

\begin{figure}[t]
\begin{center}
\includegraphics[width=0.8\linewidth]{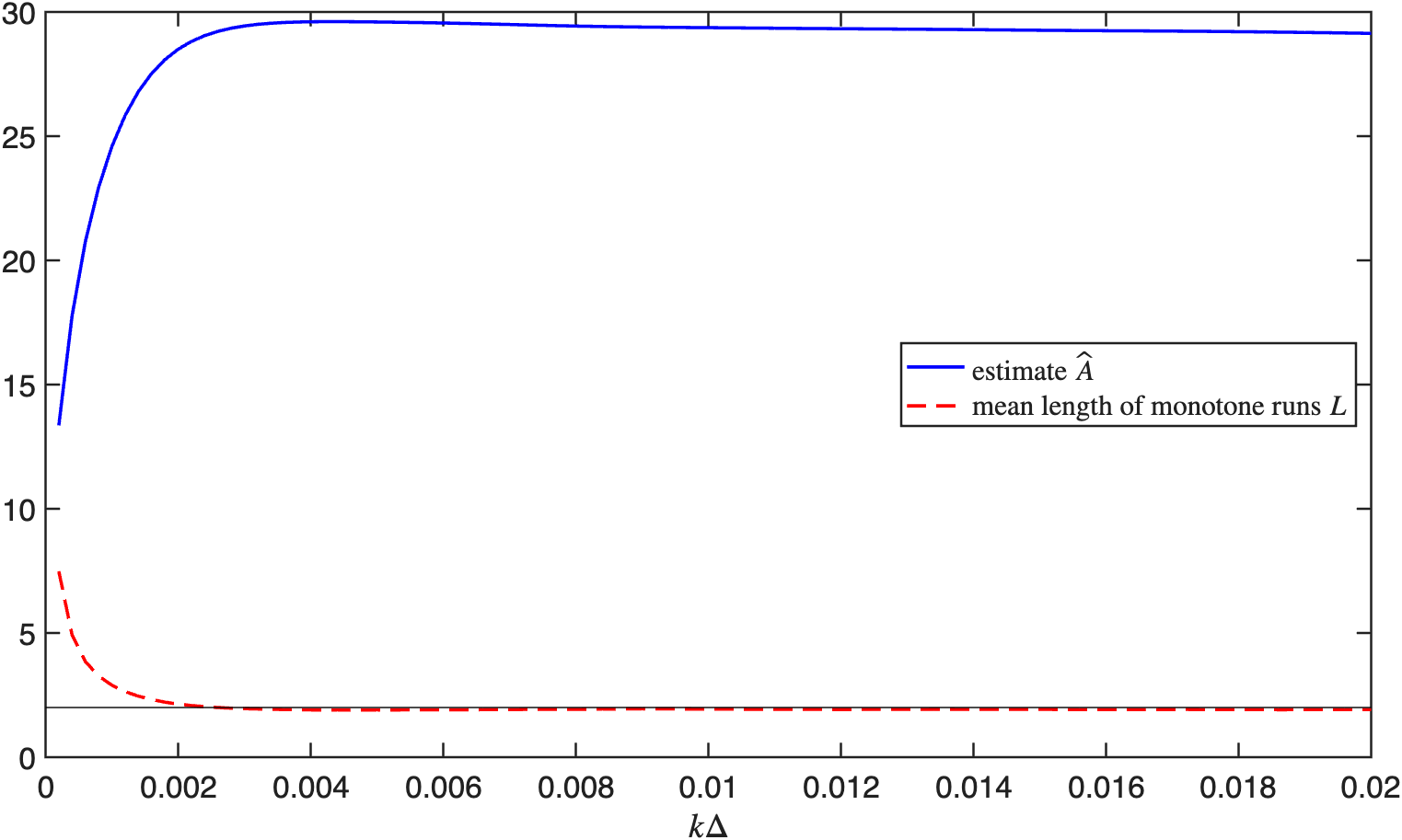} 
\end{center}
\caption{Real-world example for estimates for the diffusion constant $A$ of the basic model \eqref{eqBasicModel} and mean lengths of monotone runs in dependence on the effective step size $k \Delta$.}
\label{figAAMonFidyst}
\end{figure}

\begin{figure}[t]
\begin{minipage}[h]{0.49\linewidth}
\begin{center}
\includegraphics[width=1\linewidth]{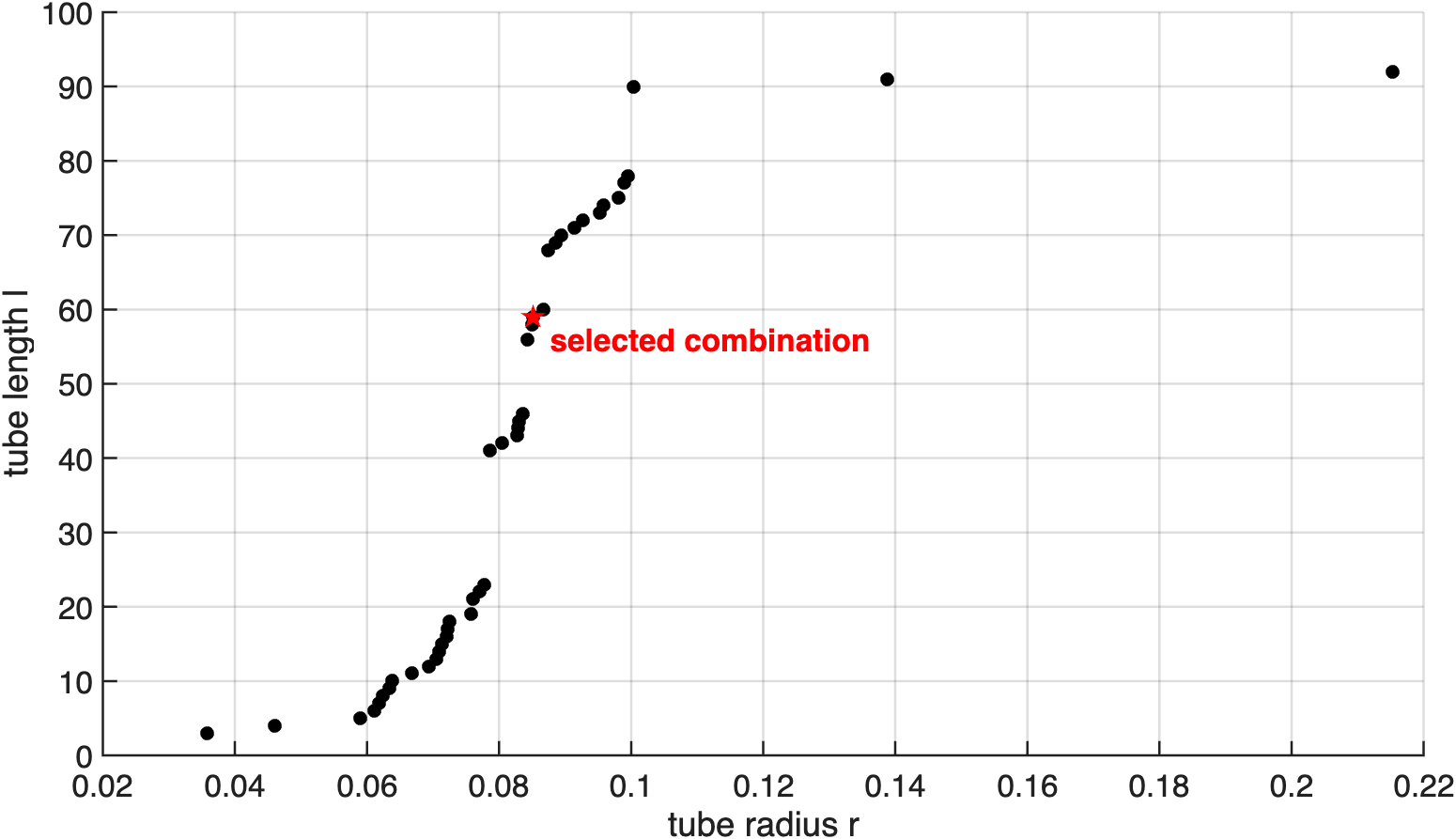} 
\end{center}
\end{minipage}
\hfill
\begin{minipage}[h]{0.49\linewidth}
\begin{center}
\includegraphics[width=1\linewidth]{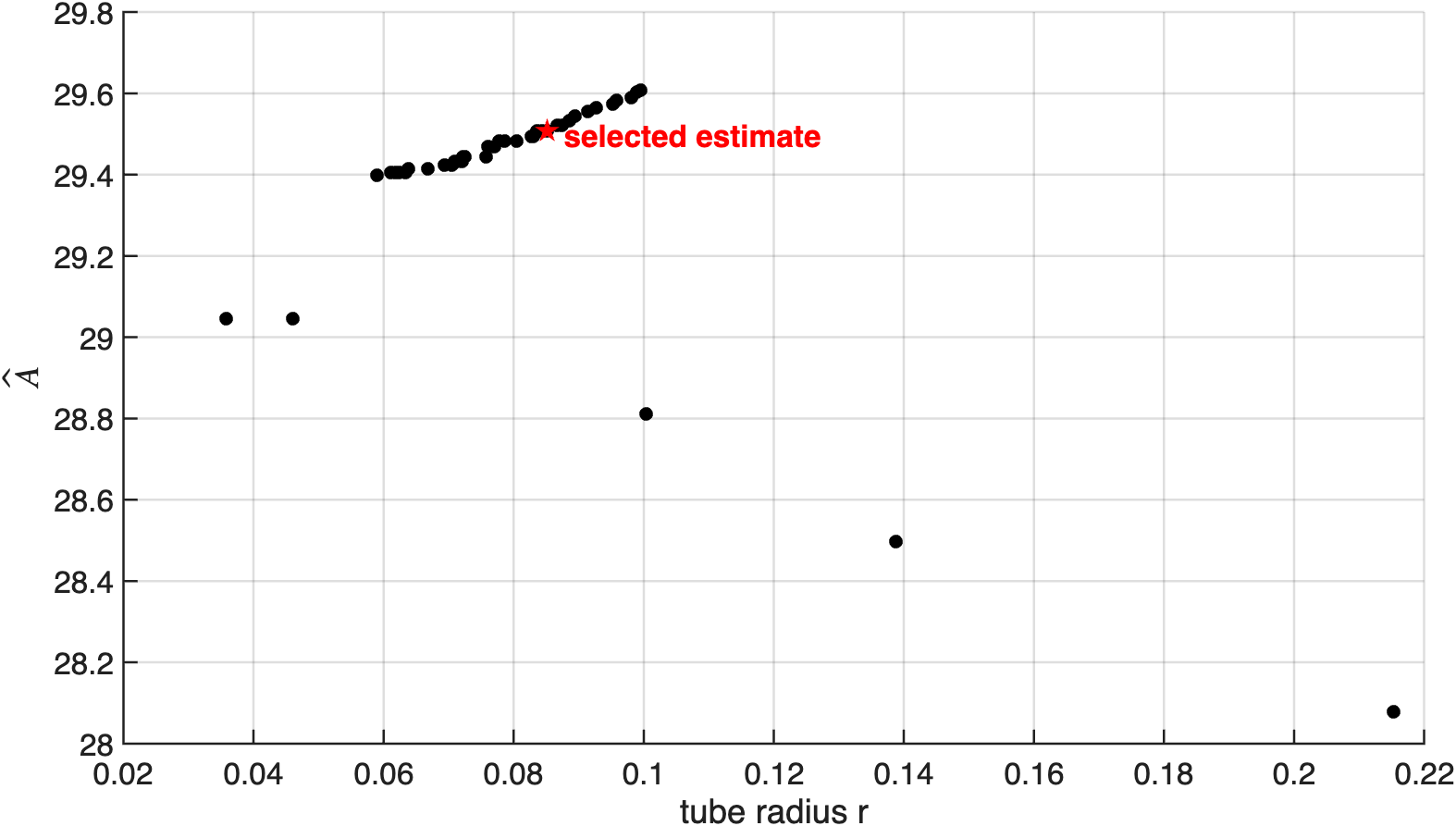} 
\end{center}
\end{minipage}
\caption{Pareto efficient combinations of tube lengths and radii around 2 for the curve of the mean length of monotone runs in the situation of Figure \ref{figAAMonFidyst} (left) and corresponding estimates for the diffusion constant $A$ in dependence on the tube radius (right).}
\label{figParetoFrontFidyst}
\end{figure}

Let us revisit the basic model
\begin{align*}
\text{d} \boldsymbol \xi_t &= \boldsymbol \tau (\alpha_t) \: \text{d}t + v \mathbf{e}_1 \: \text{d}t \\
\text{d} \alpha_t &= - \nabla V(\boldsymbol \xi_t) \cdot \boldsymbol \tau^\bot (\alpha_t) \: \text{d}t + A \: \text{d} W_t,
\end{align*}
with potential $V(\boldsymbol\xi) = \tfrac{1}{2} \boldsymbol\xi^\top \, \mathbf{C}^{-1} \boldsymbol\xi$ and $\mathbf{C} = \text{diag}(\sigma_1^2, \sigma_2^2)$. 
To estimate the throwing ranges $\sigma_1$ and $\sigma_2$, we use the idea of the existing Fraunhofer ITWM method. The estimation method discussed in \cite{Klar2009,Grothaus2014,Wegener2015} is based on the fact that the stationary distribution of $(\boldsymbol\xi, \alpha)$ is Gaussian with expectation $\mathbf{0}$ and covariance matrix $\mathbf{C}$. Obvious estimators for $\sigma_1$ and $\sigma_2$ are therefore given by the empirical standard deviations
\begin{equation*}
\sqrt{\frac{\sum_{n=0}^N {\xi_{n \Delta}^{(1)}}^2}{N+1}} \text{ and }
\sqrt{\frac{\sum_{n=0}^N {\xi_{n \Delta}^{(2)}}^2}{N+1}}.
\end{equation*}
In order to be able to motivate estimators via a stationary distribution also in the case of a moving conveyor belt, the current approach uses the large noise limit, i.e., $A \rightarrow \infty$, of the basic model. This results in the reduced model
\begin{equation*}
\text{d}\boldsymbol\xi_t = -A^{-2} \nabla V(\boldsymbol\xi_t) \, \text{d}t + v \mathbf{e}_1 \, \text{d}t + \sqrt{2} A^{-1} \, \text{d} \mathbf{W}_t,
\end{equation*}
where $\mathbf{W}$ is a two-dimensional standard Brownian Motion, for details see \cite{Bonilla2008}. It has an explicit stationary density  
\begin{equation*}
p_S(\boldsymbol\xi) \coloneq c \exp\left(-\frac{1}{2} \boldsymbol\xi^\top \mathbf{C}^{-1} \boldsymbol\xi + A^2 v \boldsymbol\xi^{(1)} \right),
\end{equation*}
where $c \in \R^+$ is the normalization constant. Using the identity
\begin{equation*}
-\frac{1}{2} \boldsymbol\xi^\top \mathbf{C}^{-1} \boldsymbol\xi + (\mathbf{C}^{-1} \boldsymbol\mu)^\top \boldsymbol\xi = -\frac{1}{2} (\boldsymbol\xi -\boldsymbol\mu)^\top \mathbf{C}^{-1} (\boldsymbol\xi- \boldsymbol\mu) + \frac{1}{2} \boldsymbol\mu^\top \mathbf{C}^{-1} \boldsymbol\mu,
\end{equation*}
where $\boldsymbol\mu \coloneq A^2 v \mathbf{C} \mathbf{e}_1$, it follows that
\begin{equation*}
p_S(\boldsymbol\xi) = \tilde{c} \exp\left( -\frac{1}{2} (\boldsymbol\xi -A^2 v \mathbf{C} \mathbf{e}_1)^\top \mathbf{C}^{-1} (\boldsymbol\xi-A^2 v \mathbf{C} \mathbf{e}_1) \right),
\end{equation*}
$\tilde{c} \in \R^+$. Thus, the stationary distribution of $\boldsymbol\xi$ from the large noise limit is Gaussian with expectation $A^2 v \mathbf{C} \mathbf{e}_1$ and covariance matrix $\mathbf{C}$. Assuming that the basic model is sufficiently close to the large noise limit, it therefore appears reasonable to center the $\boldsymbol\xi$ data and then, as in the case of the stationary conveyor belt, determine the empirical standard deviations of their components to estimate $\mathbf{C}$. Numerical experiments performed by us have shown that for a range covering all realistic combinations for $\sigma_1$, $\sigma_2$ and $A$, the stationary distribution of the basic model is well approximated by the stationary distribution of the large noise limit. For this reason, we ultimately use
\begin{equation}
\label{eqEstSigma}
\hat{\sigma}_1 \coloneq \sqrt{\frac{\sum_{n=0}^N  {(\xi_{n \Delta}^{(1)}-\bar{\xi}^{(1)}})^2}{N+1}} \text{ and }
\hat{\sigma}_2 \coloneq \sqrt{\frac{\sum_{n=0}^N  {(\xi_{n \Delta}^{(2)}-\bar{\xi}^{(2)}})^2}{N+1}},
\end{equation}
where $\bar{\xi}^{(i)} \coloneq \frac{1}{N} \sum_{n=1}^N \xi_{n \Delta}^{(i)}$, $i = 1,2$,
as estimators for the throwing ranges.

The results for a specific data set of a simulated fiber consisting of $182\,500$ data points with a step size of $\Delta = 2 \times 10^{-4}$, are shown in Figures \ref{figAAMonFidyst} and \ref{figParetoFrontFidyst}. The qualitative behavior of the curves for the mean length of monotone runs and the estimated values $\hat{A}$ of the diffusion constant is very similar to that of the curves in the simulated examples in Subsection \ref{SubsectionNumericalTests}. In particular, the curve of the mean length of monotone runs $L$ significantly exceeds the value 2 for small effective step sizes $k \Delta$ and therefore indicates the incompatibility of the available fiber data with the model. To find a suitable subsampling factor $k$ for which $L$ is close to 2, we use our method described in Algorithm \ref{AlgorithmFullProcedure} with $k_{\text{max}} = 100$. This yields the optimal subsampling factor $k^* = 34$ with a tube of length $l^* = 59$ and radius $r^* \approx 0.0851$ around 2, in which $L$ is located. The corresponding estimate of the diffusion constant is $\hat{A}^* = 29.5082$. Using \eqref{eqEstSigma}, the estimated throwing ranges in this example are $\hat{\sigma}_1 = 0.0843$ and $\hat{\sigma}_2 = 0.0452$.

\section{Conclusion and outlook}

A data-driven method for selecting a scale at which data and SDE model are compatible was presented and investigated in this paper. The analysis of monotone runs provides a suitable indicator for the rate at which the data should be subsampled to fit the model. This enables the identification of the diffusion constant of the SDE model assumed in this paper beyond specific multiscale settings and restrictive assumptions regarding practical application, such as decreasing step sizes. The method was demonstrated in a practical application using an example from the production of nonwoven textiles, in which fibers are simulated using an SDE model.

Only SDE models with additive noise were studied in this paper. It therefore remains an open question to what extent the results on the length of monotone runs on an infinitesimal scale, and thus also our method, can be applied to more general SDEs. Other open issues that were only addressed briefly in this paper include the use of alternative indicators or asymptotic confidence intervals (see Remark \ref{RemarkAlternativeIndicators}) and the modification of our method by using the extrema quadratic variation approach (see Remark \ref{RemarkExtQV}). 

\appendix

\section{Homogeneous Markov chains and ergodicity}
\label{SecAppendixErgodicity}

Let $(\mathbf{Z}_n)_{n \in \N_0}$ be an $\R^d$-valued homogeneous Markov chain on $(\Omega, \mathcal{A}, P)$ with Markov kernel $K$ and $\nu$ a probability measure on $\R^d$ such that $\mathbf{Z}_0 \sim \nu$. On the path space $((\R^d)^{\N_0}, \mathcal{B}(\R^d)^{\N_0})$ there exists a unique probability measure $\mathbbm{P}\!_{\nu}$ such that the coordinate process has the same finite-dimensional distributions as $(\mathbf{Z}_n)_{n \in \N_0}$, see, e.g., \cite[Theorem 3.1.2]{Douc2018}.

To distinguish between homogeneous Markov chains $(\mathbf{Z}_n)_{n \in \N_0}$ on $(\Omega, \mathcal{A}, P)$ with Markov kernel $K$ that differ only in their initial distributions, by $\mathbf{Z}^\nu: \Omega \rightarrow (\R^d)^{\N_0}$ we introduce the mapping $\omega \mapsto (\mathbf{Z}_n(\omega))_{n \in \N_0}$ corresponding to the initial distribution $\nu$, i.e., $\mathbf{Z}_0 \sim \nu$. Then it holds that $\mathbbm{P}\!_{\nu} = P_{\mathbf{Z}^\nu}$, see also, e.g., \cite[Theorem 1.3.4]{Douc2018}. The following lemma allows us to switch between initial distributions $\nu$ and $\delta_{\mathbf{z}}$, $\mathbf{z} \in \R^d$.

\begin{lemma}
\label{LemmaAlmostSurelyPathspace}
Let $B \in \mathcal{B}(\R^d)^{\N_0}$. Then it holds that $\mathbf{Z}^\nu \in B$ $P$-almost surely if and only if for $\nu$-almost all $\mathbf{z} \in \R^d$ it holds that $\mathbf{Z}^{\delta_{\mathbf{z}}} \in B$ $P$-almost surely.
\end{lemma}

\begin{proof}
By definition of $\mathbbm{P}\!_{\nu}$, $\mathbf{Z}^\nu \in B$ $P$-almost surely is equivalent to $\mathbbm{P}\!_{\nu}(B) = 1$.
Using \cite[Proposition 3.1.3 (ii)]{Douc2018}, we get
\begin{equation*}
1 = \mathbbm{P}\!_{\nu}(B) = \int_{\R^d} \mathbbm{P}\!_{\delta_{\mathbf{z}}}(B) \, \nu(\text{d}\mathbf{z}).
\end{equation*}
This is equivalent to
\begin{equation*}
\int_{\R^d} 1 - \mathbbm{P}\!_{\delta_{\mathbf{z}}}(B) \, \nu(\text{d}\mathbf{z}) = 1 - \int_{\R^d} \mathbbm{P}\!_{\delta_{\mathbf{z}}}(B) \, \nu(\text{d}\mathbf{z}) = 0.
\end{equation*}
Since $1 - \mathbbm{P}\!_{\delta_{\mathbf{z}}}(B) \in [0,1]$ for all $\mathbf{z} \in \R^d$, this is true if and only if $1 - \mathbbm{P}\!_{\delta_{\mathbf{z}}}(B) = 0$, i.e., $\mathbbm{P}\!_{\delta_{\mathbf{z}}}(B) = 1$, for $\nu$-almost all $\mathbf{z} \in \R^d$. By definition of $\mathbbm{P}\!_{\delta_{\mathbf{z}}}$, this is equivalent to $\mathbf{Z}^{\delta_{\mathbf{z}}} \in B$ $P$-almost surely for $\nu$-almost all $\mathbf{z} \in \R^d$.
\end{proof}

An established concept of ergodicity is based on the definition of invariant sets. For our purposes, it suffices to consider the special case of stationary processes, which is discussed, for example, in \cite{Krengel1985}. For the more general situation of dynamic systems, see, for example, \cite{Douc2018}.

\begin{definition}[Invariant sets and ergodicity]
Let $\mathbf{Z}=(\mathbf{Z}_n)_{n \in \N_0}$ be a stationary $\R^d$-valued process.
A set $A \in \mathcal{A}$ is called invariant w.r.t. $\mathbf{Z}$ if there exists some $B \in \mathcal{B}(\R^d)^{\N_0}$ such that
\begin{equation*}
A = \lbrace (\mathbf{Z}_j, \mathbf{Z}_{j+1}, \ldots) \in B \rbrace
\end{equation*}
is true for all $j \in \N_0$. Furthermore, $\mathbf{Z}$ is called ergodic if any invariant set $A \in \mathcal{A}$ satisfies $P(A) \in \lbrace 0,1 \rbrace$.
\end{definition}

Similarly to the situation of solution processes to an SDE, we call a probability measure $\mu$ an invariant measure for a Markov kernel if the associated Markov chain with initial distribution $\mu$ is stationary. The following lemma gives an equivalent characterization of ergodicity of homogeneous Markov chains.

\begin{lemma}
\label{LemmaErgodicity}
Let $K$ be a Markov kernel on $(\R^d, \mathcal{B}(\R^d))$ with invariant measure $\mu$ and let $\mathbf{Z} = (\mathbf{Z}_n)_{n \in \N_0}$ be a homogeneous Markov chain with kernel $K$ and initial distribution $\mathbf{Z}_0 \sim \mu$. Then $\mathbf{Z}$ is ergodic if and only if
\begin{equation}
\label{eqLemmaErgodicity}
\lim_{N \rightarrow \infty} \frac{1}{N} \sum_{n=0}^{N-1} \mathbbm{1}_B(\mathbf{Z}_n) = \mu(B) \quad \text{in probability} \quad \text{for all } B \in \mathcal{B}(\R^d).
\end{equation}
\end{lemma}

\begin{proof}
Let $\mathbf{Z}_0 \sim \mu$ and let $B \in \mathcal{B}(\R^d)$ be arbitrary but fixed. Since $\mathbf{Z}$ is stationary by assumption, $(\mathbbm{1}_B(\mathbf{Z}_n))_{n \in \N_0}$ is also stationary, see \cite[Corollary 4.2]{Krengel1985}. By \cite[Theorem 4.4]{Krengel1985},
\begin{equation*}
\lim_{N \rightarrow \infty} \frac{1}{N} \sum_{n=0}^{N-1} \mathbbm{1}_B(\mathbf{Z}_n) = P (\mathbf{Z}_0 \in B \, \vert \, \mathcal{J}) \quad P\text{-a.s.},
\end{equation*}
where $\mathcal{J}$ denotes the $\sigma$-algebra of invariant events of $(\mathbbm{1}_B(\mathbf{Z}_n))_{n \in \N_0}$. In particular,
\begin{equation}
\label{eqProofErgodicityInProb}
\lim_{N \rightarrow \infty} \frac{1}{N} \sum_{n=0}^{N-1} \mathbbm{1}_B(\mathbf{Z}_n) = P (\mathbf{Z}_0 \in B \, \vert \, \mathcal{J}) \quad \text{in probability}.
\end{equation}

Now, if $\mathbf{Z}$ is ergodic, then $(\mathbbm{1}_B(\mathbf{Z}_n))_{n \in \N_0}$ is also ergodic by \cite[Proposition 4.3]{Krengel1985} Thus, $P (\mathbf{Z}_0 \in B \, \vert \, \mathcal{J}) = P (\mathbf{Z}_0 \in B)$, which together with \eqref{eqProofErgodicityInProb} yields the first implication.

To show the converse, note that by \eqref{eqLemmaErgodicity} and \eqref{eqProofErgodicityInProb}, the uniqueness of the limit of convergence in probability implies $P (\mathbf{Z}_0 \in B \, \vert \, \mathcal{J}) = \mu(B)$ $P$-almost surely. Thus,
\begin{equation*}
\lim_{N \rightarrow \infty} \frac{1}{N} \sum_{n=0}^{N-1} \mathbbm{1}_B(\mathbf{Z}_n) = \mu(B) \quad P\text{-a.s.}
\end{equation*}
Let $\pi_n$ denote the $n$-th canonical coordinate mapping on $((\R^d)^{\N_0}, \mathcal{B}(\R^d)^{\N_0})$. Since $\mathbf{Z}_n(\omega) = \pi_n(\mathbf{Z}(\omega))$,
\begin{align*}
&\left\lbrace \omega \in \Omega: \lim_{N \rightarrow \infty} \frac{1}{N} \sum_{n=0}^{N-1} \mathbbm{1}_{B}(\mathbf{Z}_n(\omega)) = \mu(B) \right\rbrace \\
&\quad = \mathbf{Z}^{-1} \left( \left\lbrace \mathbf{z} \in (\R^d)^{\N_0}: \lim_{N \rightarrow \infty} \frac{1}{N} \sum_{n=0}^{N-1} \mathbbm{1}_{B}(\pi_n(\mathbf{z})) = \mu(B) \right\rbrace \right).
\end{align*}
Thus, \eqref{eqLemmaErgodicity} translates to
\begin{equation}
\label{eqLemmaErgodicityUpdated}
\lim_{N \rightarrow \infty} \frac{1}{N} \sum_{n=0}^{N-1} \mathbbm{1}_B(\pi_n) = \mu(B) \quad P_{\mathbf{Z}} \text{-a.s.} \quad \text{for all } B \in \mathcal{B}(\R^d).
\end{equation}
For a probability measure $\nu$ on $\R^d$ let $\mathbbm{P}\!_{\nu}$ denote the unique probability measure on the path space $((\R^d)^{\N_0}, \mathcal{B}(\R^d)^{\N_0})$ such that $(\pi_n)_{n \in \N_0}$ has the same finite-dimensional distributions as $\mathbf{Z}$ when $\mathbf{Z}_0 \sim \nu$. 
Now let $A \in \mathcal{A}$ be an invariant set for $\mathbf{Z}$. Then there exists a shift invariant set $B^* \in \mathcal{B}(\R^d)^{\N_0}$ such that $A = \mathbf{Z}^{-1}(B^*)$, i.e., $\mathbbm{1}_A = \mathbbm{1}_{B^*}(\mathbf{Z})$, see \cite[p.\,26]{Krengel1985}. Set $C = \lbrace \mathbf{z} \in \R^d: \mathbbm{P}\!_{\delta_{\mathbf{z}}}(B^*) = 1 \rbrace$. Since $\mathbbm{1}_{B^*}$ is a shift invariant random variable, \cite[Proposition 5.2.2 (iii)]{Douc2018} gives $\mathbbm{1}_{B^*} = \mathbbm{P}\!_{\delta_{\pi_0}}(B^*)$ $\mathbbm{P}\!_\mu$-a.s., so $\mathbbm{1}_{B^*} = \mathbbm{1}_C(\pi_0)$ $\mathbbm{P}\!_\mu$-a.s.\ by construction of $C$. Since under $\mathbbm{P}\!_\mu$ the process $(\pi_n)_{n \in \N_0}$ is stationary, the same arguments give $\mathbbm{1}_{B^*} = \mathbbm{1}_C(\pi_n) \: \mathbbm{P}\!_\mu$-a.s.\ for all $n \in \N_0$. Thus by \eqref{eqLemmaErgodicityUpdated},
\begin{equation*}
\mathbbm{1}_{B^*} = \frac{1}{N} \sum_{n=0}^{N-1} \mathbbm{1}_C(\pi_n) \xrightarrow{N \rightarrow \infty} \mu(C) \quad \mathbbm{P}\!_\mu \text{-a.s.}
\end{equation*}
and in particular
\begin{equation*}
\mathbbm{P}\!_\mu(B^*) = \mathbbm{E}_\mu(\mathbbm{1}_{B^*}) = \mathbbm{E}_\mu (\mu(C)) = \mu(C) \in \lbrace 0,1 \rbrace,
\end{equation*}
where $\mathbbm{E}_\mu$ denotes the expected value w.r.t. $\mathbbm{P}\!_\mu$.
Finally,
\begin{equation*}
P(A) = P(\mathbf{Z}^{-1}(B^*)) = P_{\mathbf{Z}}(B^*) = \mathbbm{P}\!_\mu(B^*) \in \lbrace 0,1 \rbrace,
\end{equation*}
i.e., $\mathbf{Z}$ is ergodic since $A$ was an arbitrary invariant set.
\end{proof}

For a Markov chain being the skeleton of a solution process $\mathbf{X}$ in the sense of Assumption \ref{Assumptions}, a sufficient condition for its ergodicity can be formulated based on the irreducibility of the chain, for details on this notion see, e.g., \cite[Chapter 9]{Douc2018}.

\begin{lemma}
\label{LemmaIrreducibility}
Suppose that Assumption \ref{Assumptions} holds, that there exists an invariant measure $\mu$ for \eqref{eqGenSDEIntro}, and let $\mathbf{X_0} \sim \mu$ be an initial random variable yielding the corresponding solution process $\mathbf{X}$ with $d$-th component $X$. Furthermore, let $(\mathbf{X}_{n \delta})_{n \in \N_0}$, $\delta > 0$, be irreducible. Then $(\mathbf{X}_{n \delta})_{n \in \N_0}$, $\delta > 0$ is ergodic.
\end{lemma}

\begin{proof}
Since $\mathbf{X}$ is a time-homogeneous Itô diffusion, $(\mathbf{X}_{n \delta})_{n \in \N_0}$ is a Markov chain. The irreducibility implies the uniqueness of the invariant measure, see \cite[Corollary 9.2.16]{Douc2018}, which is sufficient for the ergodicity $(\mathbf{X}_{n \delta})_{n \in \N_0}$, see \cite[Theorem 5.2.6]{Douc2018}.
\end{proof}

\bibliographystyle{alpha}
\bibliography{references}
\end{document}